\documentclass[12pt,a4paper]{amsart}

\usepackage{amsfonts,amsmath,amssymb}
\usepackage{hyperref}
\usepackage[all]{xy}

\usepackage{amssymb,amsthm,amsxtra}
\usepackage[usenames,dvipsnames]{xcolor}

\usepackage{amscd}
\usepackage{amsthm}
\usepackage{amsfonts}
\usepackage{amssymb}
\usepackage{mathrsfs}
\usepackage{enumerate}

\usepackage[usenames,dvipsnames]{xcolor}

\definecolor{dblue}{RGB}{6,69,173}
\definecolor{lblue}{RGB}{11,0,128}

\newcommand{\colorlinks}{true}

\newcommand{\linkcolor}{lblue}
\newcommand{\citecolor}{lblue}
\newcommand{\urlcolor}{dblue}

\newcommand{\linkbordercolor}{red}
\newcommand{\citebordercolor}{green}
\newcommand{\urlbordercolor}{cyan}

\usepackage{hyperref}
\hypersetup{colorlinks=\colorlinks,linkbordercolor=\linkbordercolor, linkcolor=\linkcolor, citecolor=\citecolor, urlcolor=\urlcolor,urlbordercolor=\urlbordercolor,citebordercolor=\citebordercolor}%

\let\oldtocsubsection=\tocsubsection
\let\oldtocsubsubsection=\tocsubsubsection
\renewcommand{\tocsubsection}[2]{\hspace{1em}\oldtocsubsection{#1}{#2}}
\renewcommand{\tocsubsubsection}[2]{\hspace{2em}\oldtocsubsubsection{#1}{#2}}

\newcommand{\hrefHid}[2]{
\hypersetup{urlbordercolor={1 1 1}}%
\hypersetup{urlcolor=black}%
\href{#1}{#2}%
\hypersetup{urlbordercolor=\urlbordercolor}%
\hypersetup{urlcolor=\urlcolor}%
}

\newcommand{\inhref}[2]{\hyperref[#1]{#2}}

\newcommand{\inhrefHid}[2]{%
\hypersetup{linkbordercolor={1 1 1}}%
\hypersetup{linkcolor=black}%
\inhref{#1}{#2}%
\hypersetup{linkbordercolor=\linkbordercolor}%
\hypersetup{linkcolor=\linkcolor}%
}

\newcommand{\defHref}[3]{\newcommand{#1}[1][#3]{\href{#2}{##1}}}
\newcommand{\defInhref}[3]{\newcommand{#1}[1][#3]{\inhref{#2}{##1}}}
\newcommand{\defHrefHid}[3]{\newcommand{#1}[1][#3]{\hrefHid{#2}{##1}}}
\newcommand{\defInhrefHid}[3]{\newcommand{#1}[1][#3]{\inhrefHid{#2}{##1}}}

\newcommand{\defHrefBoth}[3]{%
\expandafter\defHrefHid \csname #3Hid\endcsname {#1}{#2}%
\expandafter\defHref \csname #3Vis\endcsname {#1}{#2}%
}
\newcommand{\defInhrefBoth}[3]{%
  \expandafter\defInhrefHid \csname #3Hid\endcsname {#1}{#2}%
  \expandafter\defInhref \csname #3Vis\endcsname {#1}{#2}%
}

\newcommand{\defHrefBothVis}[3]{%
\defHrefBoth{#1}{#2}{#3}%
\expandafter\defHref \csname #3\endcsname {#1}{#2}%
}
\newcommand{\defInhrefBothVis}[3]{%
  \defInhrefBoth{#1}{#2}{#3}%
  \expandafter\defInhref \csname #3\endcsname {#1}{#2}%
}

\newcommand{\defHrefBothHid}[3]{%
\defHrefBoth{#1}{#2}{#3}%
\expandafter\defHrefHid \csname #3\endcsname {#1}{#2}%
}
\newcommand{\defInhrefBothHid}[3]{%
  \defInhrefBoth{#1}{#2}{#3}%
  \expandafter\defInhrefHid \csname #3Vis\endcsname {#1}{#2}%
}

\newcommand{\term}[2]{%
\label{#2}%
\emph{#1}%
 \globaldefs =1%
\defInhrefBothVis{#2}{#1}{#2}%
 \globaldefs =0%
}

\newcommand{\simpAr}[2][r]{%
\ar@{}[#1]|-*[@]_{#2}%
}

\newtheorem*{theorem*}{Theorem}
\newtheorem{lemma}{Lemma}[subsection]
\newtheorem{proposition}[lemma]{Proposition}
\newtheorem*{proposition*}{Proposition}

\newtheorem{theorem}[lemma]{Theorem}

\newtheorem{corollary}[lemma]{Corollary}

\newtheorem*{conjecture*}{Conjecture}
\newtheorem*{lemma*}{Lemma}

\newtheorem{thm}[lemma]{Theorem}
\newtheorem{prop}[lemma]{Proposition}
\newtheorem{lem}[lemma]{Lemma}
\newtheorem{cor}[lemma]{Corollary}

\newtheorem{introtheorem}{Theorem}

\newtheorem{introthm}[introtheorem]{Theorem}
\newtheorem{introlem}[introtheorem]{Lemma}

\theoremstyle{remark}
\newtheorem{remark}[lemma]{Remark}

\newtheorem*{remark*}{Remark}
\newtheorem*{remarks*}{Remarks}
\newtheorem{rem}[lemma]{Remark}

\newtheorem*{examples*}{Examples}
\newtheorem*{example*}{Example}
\newtheorem{definition}[lemma]{Definition}
\newtheorem{notation}[lemma]{Notation}
\newtheorem{defn}[lemma]{Definition}
\newtheorem{notn}[lemma]{Notation}

\baselineskip 18pt \textwidth 16cm  \theoremstyle{plain}

\newcommand{\Hom}{\operatorname{Hom}}

\newcommand{\A}{\mathbb{A}}

\newcommand{\eps}{\varepsilon}

\newcommand{\id}{\operatorname{Id}}

\renewcommand{\Im}{\operatorname{Im}}

\newcommand{\Z}{{\mathbb Z}}
\newcommand{\Q}{{\mathbb Q}}
\newcommand{\R}{{\mathbb R}}

\newcommand{\C}{{\mathbb C}}

\newcommand{\pr}{{\operatorname{pr}}}

\newcommand{\End}{\operatorname{End}}

\newcommand{\G}{{\mathcal G}}

\newcommand{\Fre}{{Fr\'{e}chet \,}}

\newcommand{\et}{{\'{e}tale }}

\newcommand{\g}{{\mathfrak{g}}}

\newcommand{\Supp}{\mathrm{Supp}}
\newcommand{\supp}{\mathrm{Supp}}

\newcommand{\Sc}{{\mathcal S}}

\newcommand{\WF}{\operatorname{WF}}
\newcommand{\cF}{\mathcal{F}}

\newcommand{\fpt}[1]{{{#1}(F)}}

\newcommand{\Rami}[1]{{{#1}}}
\newcommand{\RamiA}[1]{{{#1}}}

\newcommand{\RamiD}[1]{{{#1}}}
\newcommand{\RamiE}[1]{{{#1}}}
\newcommand{\RamiF}[1]{{{#1}}}
\newcommand{\RamiG}[1]{{{#1}}}

\newcommand{\RamiI}[1]{{{#1}}}
\newcommand{\RamiJ}[1]{{{#1}}}
\newcommand{\RamiK}[1]{{{#1}}}
\newcommand{\RamiL}[1]{{{#1}}}
\newcommand{\RamiM}[1]{{{#1}}}

\newcommand{\RamiO}[1]{{{#1}}}
\newcommand{\RamiP}[1]{{{#1}}}
\newcommand{\RamiR}[1]{{{#1}}}
\newcommand{\RamiS}[1]{{{#1}}}
\newcommand{\RamiT}[1]{{{#1}}}
\newcommand{\RamiU}[1]{{{#1}}}
\newcommand{\RamiV}[1]{{{#1}}}
\newcommand{\RamiW}[1]{{{#1}}}
\newcommand{\RamiX}[1]{{{#1}}}
\newcommand{\RamiY}[1]{{#1}}

\newcommand{\RamiAA}[1]{{{#1}}}
\newcommand{\RamiAB}[1]{{{#1}}}
\newcommand{\RamiAC}[1]{{{#1}}}
\newcommand{\RamiAD}[1]{{{#1}}}

\newcommand{\RamiAE}[1]{{{#1}}}
\newcommand{\RamiAF}[1]{{{#1}}}
\newcommand{\RamiAG}[1]{{{#1}}}
\newcommand{\RamiAH}[1]{{{#1}}}
\newcommand{\RamiAI}[1]{{{#1}}}
\newcommand{\RamiAJ}[1]{{{#1}}}
\newcommand{\RamiAK}[1]{{{#1}}}
\newcommand{\RamiAL}[1]{{{#1}}}
\newcommand{\RamiAM}[1]{{{#1}}}
\newcommand{\RamiAN}[1]{{{#1}}}
\newcommand{\RamiAO}[1]{{{#1}}}

\newcommand{\RamiAQ}[1]{{{#1}}}

\newcommand{\RamiAS}[1]{{{#1}}}

\newcommand{\DrinA}[1]{{{#1}}}
\newcommand{\DrinB}[1]{{{#1}}}
\newcommand{\DrinC}[1]{{{#1}}}
\newcommand{\DrinD}[1]{{{#1}}}

\newcommand{\DrinF}[1]{{{#1}}}
\newcommand{\DrinG}[1]{{{#1}}}
\newcommand{\DrinH}[1]{{{#1}}}
\newcommand{\DrinI}[1]{{#1}}

\newcommand{\RamiCP}[1]{{{#1}}}
\newcommand{\RamiCPA}[1]{{{#1}}}

\newcommand{\RamiCPD}[1]{{{#1}}}
\newcommand{\RamiCPE}[1]{{{#1}}}

\newcommand{\RamiQ}[1]{{}}
\newcommand{\RamiQB}[1]{{}}

\newcommand{\RamiQAB}[1]{}
\newcommand{\RamiAbs}[1]{}

\newcommand{\oldFT}[1]{}

\newcommand{\oldWFHol}[1]{}
\newcommand{\newWFHol}[1]{{#1}}

\newcommand{\Crit}{\operatorname{Crit}}
\newcommand{\mmu}{\mu}
\newcommand{\mmup}{\mu'}
\newcommand{\pii}{\pi_\infty}
\newcommand{\tD}{\widetilde{D}}

\usepackage{varioref}

\begin{document}

\author{Avraham Aizenbud}
\address{Avraham Aizenbud,
Massachusetts Institute of Technology
Department of Mathematics
Cambridge, MA 02139 USA.}
\email{aizenr@gmail.com}
\urladdr{\url{http://math.mit.edu/~aizenr/}}

\author{Vladimir Drinfeld}
\address{Vladimir Drinfeld, Department of Mathematics, University of Chicago.}
\email{drinfeld@math.uchicago.edu}


\title{The wave front set of the Fourier transform of algebraic measures}
\keywords{Wave front set, Fourier transform, Holonomic distributions, $p$-adic, Desingularization \RamiQ {Can you review the keywords?}\\
\indent
 MSC Classes: \RamiCP{46F, 46F10.}\RamiQ{Can you review the MSC Classes?
Can you suggest additional classes? Here is the description of the classes I have put:\\
46-xx  Functional analysis. 46Fxx Distributions, generalized
functions, distribution spaces. 46F10 Operations with distributions.
}
}
%
%
%
%
%
%
%
%
%
%
\maketitle
\begin{abstract}
We study the Fourier transform of \RamiD{the absolute value of a} polynomial on a finite-dimensional vector
 space over a local field of characteristic 0. We prove that this transform is  smooth on an open dense set.

\RamiV{We prove this result for the Archimedean and the non-Archimedean case in a uniform way.}
The Archimedean case was proved in \cite{Ber_FT}.
\RamiV{T}he non-Archimedean case  was proved in \cite{HK} \RamiD{and \cite{CL,CL2}}.
 Our method is different from those described in \RamiV{\cite{Ber_FT,HK,CL,CL2}}.
 It is based on Hironaka's desingularization theorem, unlike  \cite{Ber_FT} \RamiM{which} is based on the theory of D-modules
and \RamiV{\cite{HK,CL,CL2}} \RamiM{which} is based on model theory.

Our method \RamiV{ also gives} bounds on the open dense set where the Fourier transform is  {smooth} and moreover, on the wave front set of the Fourier transform.
These bounds are explicit in terms of resolution of singularities and field-independent.

We also  prove the same results on the Fourier transform of more general \RamiE{measures} of algebraic origins.

\end{abstract}
\setcounter{tocdepth}{3}
\tableofcontents

\section{Introduction}

\subsection{Main results \DrinI{in the non-Archimedean case}}   \label{ss:main}
\begin{introtheorem} \label{thm:main_a}
Let $F$ be a non-Archimedean local field of characteristic $0$
 (i.e. a finite extension of the field  of p-adic numbers $\Q_p$).
Let $W$ be a finite-dimensional vector space over $F$. Let \RamiI{$X$ be a smooth algebraic variety over
$F$, let $\phi: X \to W$ be a proper map}
and $\omega$ a \DrinD{regular (algebraic)} top differential form on  {$X$}.
Let $|\omega|$ be the measure on $X$ corresponding to $\omega$ and  {$\phi_*(|\omega|)$ its direct image (which is a measure on $W$).}
Then there exists  {a dense Zariski open} subset $U\subset W^*$
 such that the restriction to $U$
of the Fourier transform of $\phi_*(|\omega|)$ \RamiE{is locally constant}.
\end{introtheorem}

\begin{examples*}
$ $
\DrinD{
\begin{itemize}
 \item  Let $X\subset W$ be a smooth closed subvariety and $\omega$ a regular top differential form on
$X$. Consider $|\omega|$ as a measure on $W$. Applying Theorem~\ref{thm:main_a} to the embedding
$\phi :X\hookrightarrow W$\RamiM{,} we see that the Fourier transform of $|\omega|$ is smooth on a dense Zariski open subset.

 \item  More generally, let $X\subset W$ be any closed subvariety and $\omega$ a rational top differential form on $X$. Suppose that for some resolution of singularities $p:\hat X\to X$\RamiM{,} the pullback $p^*(\omega )$ is
regular. Then one can consider $|\omega|$ as a measure on $W$. Its Fourier transform is smooth on a dense Zariski open subset (to see this, apply Theorem~\ref{thm:main_a} to the composition
$\hat X\to X\hookrightarrow W$).
\end{itemize}
}
\end{examples*}

 \DrinD{We deduce Theorem~\ref{thm:main_a} from the following theorem, which says that the singularities of the Fourier transform of  $\phi_*(|\omega|)$ on the whole $W^*$ are ``not too bad".}
\DrinD{
\begin{introtheorem} \label{thm:main_b}
In the situation of Theorem~\ref{thm:main_a}\RamiM{,} the wave front set of the Fourier transform of
$\phi_*(|\omega|)$ is contained in
\RamiX{an isotropic}
algebraic subvariety\footnote{An algebraic subvariety of a symplectic algebraic manifold is said to be \RamiX{isotropic}
if its nonsingular part is. \RamiX{Note that every algebraic isotropic subvariety of the co-tangent bundle of an algebraic manifold $M$ which is stable with respect to homotheties along the co-tangent space is contained in a union of co-normal bundles of submanifolds of $M$, and in particular in a Lagrangian subvariety. See \S\S\ref{sec:lag} for more details. }} of $T^* (W^*)=W\times W^*$.
\end{introtheorem}
}

The notion of wave front set is is recalled in Appendix~\ref{app:WF} and \S\ref{sssec:wfs}.
It was introduced by L. H\"ormander \cite{Hor} to study the singularities of functions on a real manifold $M$ microlocally (roughly speaking, to study them not only in space but also with respect to Fourier transform at each point of $M$). Later D. B. Heifetz \cite{Hef} defined this notion if $M$ is a $p$-adic manifold.





\begin{rem}
Theorem~\ref{thm:main_b} can be considered as a $p$-adic analog of the following theorem of J.~Bernstein \cite{Ber_FT}:
 if $F$ is Archimedean then the Fourier transform of $\phi_*(|\omega|)$ is a holonomic distribution.
 \end{rem}

\DrinD{We also prove the following more general, relative version of Theorem~\ref{thm:main_b}.}


\begin{introtheorem} \label{thm:main_bb}
Let $F$ be a non-Archimedean field of characteristic $0$. Let $W$
be a finite-dimensional $F$-vector space and $X,Y$ be smooth algebraic manifolds over $F$. Let
$\phi: X \to Y \times W$ be a proper map
 and let $\omega$ be a \RamiR{regular} top differential form on $X$.
Then the wave front set of the partial Fourier transform of
$\phi_*(|\omega|)$ with respect to $W$
is contained in
\RamiX{an isotropic}
 algebraic subvariety of $T^*(Y\times W^*)$.
\end{introtheorem}

\subsection{\DrinI{The Archimedean analog of Theorem~\ref{thm:main_bb}}}
\DrinI{
We have to take in account that in the  Archimedean case the Fourier transform is defined not for general distributions, but only for Schwartz distributions.  Similarly, the partial Fourier transform is defined for distributions that are \emph{partially Schwartz} along the relevant vector space (a precise definition of partially Schwartz distribution can be found in
\S\S\ref{ssec:arc_dist} below).

\begin{introtheorem} \label{thm:main_C_Arch}
Let $F$ be an Archimedean local field (i.e. $\R$ or $\C$). Let everything else be as in Theorem \ref{thm:main_bb}. Then
\begin{enumerate}[(i)]
\item \label{thm:main_C_Arch:1}
the distribution $\phi_*(|\omega|)$ is partially Schwartz along $W$ (so its partial Fourier transform with respect to $W$ is well-defined);
\item  \label{thm:main_C_Arch:2}
the wave front set of the partial Fourier transform of
$\phi_*(|\omega|)$ with respect to $W$
is contained in
\RamiX{an isotropic}
algebraic subvariety of $T^*(Y\times W^*)$.
\end{enumerate}
\end{introtheorem}
}
\RamiW{
\begin{rem}
In fact, the distribution  $\phi_*(|\omega|)$ is Schwartz on the entire space and not only along $W$, but in order to prove it we need to define what it means, and we prefer not to do it in this paper.
\end{rem}
}
\subsection{\RamiV{Stronger versions}}
\RamiV{Our proof of theorem \ref{thm:main_bb} can give an explicit (in terms of resolution of singularities) description of
\RamiX{an isotropic}
 variety that contains the wave front set of the partial Fourier transform of
$\phi_*(|\omega|)$. We provide such description in Theorem \ref{exp} (and an analogous description for Theorem \ref{thm:main_b} in Corollary \ref{cor:exp}). This description implies that this
\RamiX{isotropic}
variety ``does not actually depend" on the local field $F$ and is stable under homotheties
\DrinI{in $W^*$}. Namely we have the following theorem:
\begin{introtheorem} \label{thm:futerO}
Let $K$ be a characteristic $0$ field and $W$ a finite-dimensional $K$-vector space. Let $X,Y$ be smooth algebraic manifolds over $K$. Let $\phi: X \to Y \times W$ be a proper map,
 $\omega$ a \RamiR{regular} top differential form on $X$.

Then there exists 
\RamiX{an isotropic}
algebraic subvariety $L \subset T^*(Y\times W^*)$ such that
\begin{enumerate}[(i)]
\item  \label{thm:futer:1}    $L$ is stable with respect to the action of the multiplicative group
\DrinI{on $T^*(Y\times W^*)$ that comes from its action on $W^*$};
\item  \label{thm:futer:2}     \DrinI{f}or any embedding of $K$ into any local field $F$ (Archimedean or not), 
the wave front set of the partial Fourier transform of
 $(\phi_F)_*( |\omega_F|)$ is contained in $L(F)$.
\end{enumerate}

Here $L(F)\subset T^*(Y\times W^*)(F)$ is the set of $F$-points of $L$ and $\omega_F$  is obtained from  $\omega$ by extension of scalars from $K$ to $F$, and $|\omega_F|$ is the corresponding measure on $X(F)$.
\end{introtheorem}
}
\begin{remark*}
The Fourier transform depends on the choice of a nontrivial additive character
 $\psi :F\to\C^{\times}$. But if $L$ satisfies (i) and has property (ii) for some $\psi$, then (ii) holds
 for any $\psi$.
\end{remark*}



\DrinI{We will show that the following variant of Theorem \ref{thm:futerO} easily follows from
Theorem \ref{thm:futerO} itself.}\footnote{\DrinI{More precisely, Theorem~\ref{thm:futer} for
 $\phi :=X \to Y\times W$ and $p:X\to K$ follows from Theorem \ref{thm:futerO} for
  $\phi\times p :X \to Y\times W \times K$. }
}



\RamiV{
\begin{introthm}  \label{thm:futer}
In the situation of Theorem  \ref{thm:futerO} let $p$ be a regular function on $X$.
Then there exists
\RamiX{an isotropic}
algebraic subvariety $L \subset T^*(Y\times W^*)$ such that
 for any embedding of $K$ into any local field $F$ \DrinI{and any nontrivial additive character
 $\psi :F\to\C^{\times}$},
the wave front set of the partial Fourier transform of
 $(\phi_F)_*(\DrinI{(}\psi \circ p_F\DrinI{)} \cdot |\omega_F|)$ is contained in $L(F)$.
Here the Fourier transform is performed using the same $\psi$, and $p_F$ is obtained from $p$ by extension of scalars from $K$ to $F$.
\end{introthm}
Again, in order for this theorem to make sense for \DrinI{Archimedean} $F$, we will prove the following lemma:
\begin{introlem} \label{lem:Sch2}
In the notations of \RamiW{Theorem} \ref{thm:futer},
let $F$ be an Archimedean local field, with an embedding $K \hookrightarrow F$.  Then the distribution
$(\phi_F)_*(\DrinI{(}\psi \circ p_F\DrinI{)} \cdot |\omega_F|)$ is partially Schwartz along $W$.
\end{introlem}

}

\begin{example*}
Let $p$  be a polynomial on an $F$-vector space $W$.
\DrinI{Theorem \ref{thm:futer} implies that the Fourier transform of the function $x\mapsto\psi (p(x))$ is smooth on a dense Zariski open  subset.}
\end{example*}

\subsection{Method of the proof and comparison with related results}
$ $

The Archimedean counterpart of Theorem \ref{thm:main_a} above was proved by J.~Bernstein \cite{Ber_FT} using
D-module theory.
In the non-Archimedean case 
Theorem \ref{thm:main_a} is one of many results proved by Hrushovski - Kazhdan \cite{HK} and Cluckers - Loeser \cite{CL,CL2} using model theory.

\DrinI{
In this work we give a proof of Theorems \ref{thm:main_a}-\ref{thm:main_C_Arch}
based on Hironaka's desingularization theorem. The proof  is simple and effective modulo desingularization
and treats Archimedean and non-Archimedean local fields in a uniform way. Since the proof is effective it also
yields Theorems \ref{thm:futerO}-\ref{thm:futer}, which seem to be new.}
Note that although Hironaka's desingularization theorem is far from being elementary,
it now has understandable proofs (e.g., see ~\cite{Hir_mod}).





 The present paper is not the first time when Hironaka's theorem is used to replace D-module theory  in the non-Archimedean case. A well-known example is one of the earliest applications of the theory of D-modules --  the regularization and analytic continuation of the distribution $p^\lambda$ where $p$ is a polynomial and
 $\lambda$ is a complex number (see \cite{ber_pl}).
 This result has an
 alternative proof based on Hironaka's theorem, which is valid both in the Archimedean and the
non-Archimedean
cases, see \cite{BG} and \cite{Ati}.

A few years ago D.~Kazhdan pointed out to us that surprisingly, Theorem~\ref{thm:main_a} does not follow immediately from Hironaka's theorem. 
However, we show that Theorem~\ref{thm:main_a} and its generalizations involving wave front sets
(Theorems \ref{thm:main_b}-\ref{thm:futer})
 follow from  Hironaka's theorem after some work.

Using wave front sets to deduce Theorem~\ref{thm:main_a} and its Archimedean counterpart from Hironaka's theorem seems very natural to us. First, wave front sets were introduced by L.~H\"ormander precisely to treat analytic problems of this type.
Second, the technique of wave front sets is ``field-independent". Other reasons are explained in Section~\ref{ss:idea} below.

\RamiAQ{\begin{rem}
A short account of the present work is given in \cite{Dri}. It includes a sketch of the proof of the main
results with emphasis on the main ideas (which are very simple). The reader may prefer to read  \cite{Dri} before reading the complete proof.
\end{rem}}

\subsection{Idea of the proof}  \label{ss:idea}
$ $

\DrinD{Theorem~\ref{thm:main_a} is deduced from Theorem \ref{thm:main_bb}. The latter
has}
two advantages:
\begin{enumerate}
  \item Since we are discussing the wave front set, Theorem \ref{thm:main_bb} is more flexible with respect to changes of $X$ and $Y$.
 \item Since we are discussing a relative version, Theorem \ref{thm:main_bb} can be  
\DrinD{approached} locally \DrinD{with respect to \RamiI{$Y$}}.
\end{enumerate}
Using those facts\RamiM{,} we can reduce Theorem \ref{thm:main_bb}  to the special case (see Proposition \ref{thm:main_c})
when the map $X \to Y$ is an \DrinD{open} embedding.
 Furthermore, using Hironaka's theorem we can assume \DrinD{that $\omega$ and $\phi$ behave
 ``nicely" in the neighborhood of $Y-X$.}


Using (1) and (2) again\RamiM{,} we can reduce further (see Lemma \ref{lem:red_to_dim_1}) to the case when $W$ is
1-dimensional.
By localizing the problem on $Y$\RamiM{,} we reduce Proposition \ref{thm:main_c} to
\DrinD{a simple local model, which has  a symmetry with respect to an action of a large torus.
This symmetry allows to prove  Proposition \ref{thm:main_c} for the local model.}

\subsection{Structure of the paper}
$ $
In \S\ref{sec:prel} we will fix notations and give the necessary preliminaries for the paper.
In  {\S\RamiAM{\S}\ref{ssec:ag} we recall two algebro-geometrical tools used in this paper.}
Namely, in \S\RamiAM{\S}\RamiAM{\S}\ref{sssec:rs} we  review Hironaka's theory of resolution of singularities
(see  \cite{Hir}, or \cite{Hir_mod} for a more recent overview), and
 in \S\RamiAM{\S}\RamiAM{\S}\ref{sssec:nag}  we recall Nagata's compactification theorem.
In \S\RamiAM{\S}\ref{ssec:dist} we review the theory of distributions and in particular, the notion of the wave front
 set. Most of the results there are from \cite{Hef} and  \cite{Aiz}. {The rest we provide in Appendix \ref{app:WF}.}

In \S\ref{sec:WF-Holonomic} we introduce the notion of WF-holonomic distributions and state some of its basic
properties. This notion can be viewed as a partial analytic counterpart of the algebraic notion of  holonomic
 distributions, which is defined via the theory of D-modules. We use this notion in order to formulate our main result. \RamiAM{In \S\S\ref{sec:lag}, we recall the basic facts from symplectic geometry that we use in order to work with WF-holonomic distributions. We provide proofs and references for those results in Appendix \ref{app:symp_geom}. }

{In  \S\ref{sec:pf_main_bb}-\ref{sec:pf_futer} we prove the main results of the paper in the non-Archimedean case.

In  \S\ref{sec:pf_main_bb} we prove  Theorem \ref{thm:main_bb} (which implies
Theorem \ref{thm:main_b} and Theorem \ref{thm:main_a}).
}


\RamiAM{In \S\ref{sec:ex_boun} we {prove} Theorem \ref{exp}, Corollary \ref{cor:exp} and  Corollary  \ref{c:explcit_open},  which are ``explicit" versions of Theorems \ref{thm:main_bb}, \ref{thm:main_b} and \ref{thm:main_a} respectively {(e.g., Theorem~\ref{thm:main_a} claims the existence of a dense open $U$ on which a certain distribution is smooth, while  Corollary  \ref{c:explcit_open} provides a concrete $U$ with this property).} We also explain how Theorem \ref{exp} implies Theorem \ref{thm:futerO}.

In \S\ref{sec:pf_futer} we deduce Theorem \ref{thm:futer} from Theorem \ref{thm:futerO}.

In \S\ref{sec:arc} we explain how to adapt the proofs {from \S\ref{sec:pf_main_bb}-\ref{sec:pf_futer}} 
for the Archimedean case.}




{In \RamiAM{Appendix} \ref{app:WF} we
\RamiAM{elaborate on}
 the results stated in  \S\S\S \ref{sssec:wfs}.}

\RamiAM{In Appendix  \ref{app:symp_geom} we elaborate on the results stated in  \S\S \ref{sec:lag}}
\RamiD{
\subsection{Acknowledgments}
We thank Patrick Gerard for a useful reference and Vladimir Hinich, for fruitful discussions.
We thank Joseph Bernstein, Ehud Hrushovski, David Kazhdan and Michael Temkin for their useful remarks.
\RamiAM{We would also like to thank Inna Entova-Aizenbud for proofreading the drafts for this work.}

Part of the work on this paper was done while Avraham Aizenbud participated in the program
 ``Analysis on Lie Groups" at the Max Planck Institute for Mathematics (MPIM) in Bonn.
Avraham Aizenbud is also partly supported by NSF grant DMS-1100943.
Vladimir Drinfeld is partly supported by NSF grant DMS-1001660.
}

\section{Preliminaries} \label{sec:prel}
\subsection{\RamiM{Notation} and conventions}\label{ssec:not}
$ $
\RamiAM{Below is a list of notations {and conventions} throughout the paper. The somewhat nonstandard among them are (\ref{it:notstd1}), (\ref{it:notstd2}), (\ref{it:notstd3}).
}
%
\subsubsection{\RamiAM{The local field $F$}}\label{sssec:not_f}
\begin{enumerate}
 \setcounter{enumi}{0}
\item\label{it:notstd1}  We fix a local  field $F$ of characteristic $0$. \RamiT{It will be non-Archimedean in the entire paper except \S\ref{sec:arc} and the Appendices.}
\RamiU{\item We always {equip $F$ with the \emph{normalized}} absolute value 
(this is the multiplicative {quasi-}character $x\mapsto |x|$ 
given by the action on Haar measures). }
\item \label{it:notstd2} We fix \RamiAG{a non-trivial} additive character \RamiE{$\psi:F \to \C^{\times}$}.
\end{enumerate}
\subsubsection{\RamiAM{Varieties and manifolds}}
$ $
\RamiQAB{the numbering shift is manual. Should be checked in the end}
\begin{enumerate}
 \setcounter{enumi}{3}

\item All the algebraic varieties and analytic varieties which we consider are \RamiE{reduced, separated and} defined over $F$.
\item {We will treat $F$-algebraic varieties as $\bar{F}$-algebraic ones equipped with an $F$-structure.}
\item  {We will treat $F$-vector spaces both as algebraic varieties and analytic varieties.} 
\item When we say ``an analytic variety", we mean an $F$-analytic variety in the classical sense of \cite{ser} and not in the sense of rigid geometry or Berkovich geometry.
\item For an algebraic variety $X$,  we will denote by $\fpt{X}$ the set of $F$ points of $X$
 considered as an analytic variety (and, in particular, as a topological space).
 \DrinD{By abuse of notation, the map $X(F) \to Y(F)$ corresponding to a morphism of algebraic varieties
 $\phi:X \to Y$ will also be denoted by $\phi$.}
\item We will use the word ``manifold" to indicate smoothness, e.g. ``algebraic manifold" will mean smooth algebraic variety.
\item When we want to speak in general about algebraic and analytic varieties or manifolds, we will just say variety or manifold.
\RamiR{\item We will use the word ``regular" only in the sense of algebraic geometry and not in the sense of analytic geometry.}
\item We will usually use the same notation for a vector bundle and its total space.
\item \RamiAN{For a vector bundle $E$ over a manifold $X$, we will identify $X$ with the zero section inside $E$.}
\end{enumerate}
\subsubsection{\RamiAM{The \RamiAN{(co)}tangent \RamiAN{and the (co)normal} bundle} 
}
$ $
\RamiQAB{the numbering shift is manual. Should be checked in the end}

\begin{enumerate}
 \setcounter{enumi}{13}
  \item For a manifold $X$\RamiM{,} we denote by  $T X=T (X)$ and  $T^* {X}=T^* {(X)}$ the tangent and co-tangent bundles, respectively. For a point $x \in X$, we denote by $T_x X=T_x(X)$ and  $T^*_x X=T^*_x(X)$ the tangent and co-tangent spaces, respectively.
\item \label{it:notstd3} For a (locally closed) submanifold $Y \subset X$, we denote by
$N_Y^X:=(T_X|_Y)/T_Y$ and $CN_Y^X:=(N_Y^X)^*$ the
normal and co-normal bundle to $Y$ in $X$, respectively.
\end{enumerate}
\subsubsection{\RamiAM{Group {and Lie algebra} action{s}}}

$ $
\RamiQAB{the numbering shift is manual. Should be checked in the end}

\begin{enumerate}
 \setcounter{enumi}{15}
\item For a group $G$ acting on a set $X,$ and a point $x \in X,$ we denote by $Gx$ or by $G(x)$ the orbit of $x$ and by $G_x$ the stabilizer of $x$.
%
\item An action of a Lie algebra $\g$ on a
manifold $M$ is a Lie algebra homomorphism from $\g$ to the Lie algebra of vector fields on $M$.
Note that an action of an ({analytic or} algebraic) group on $M$ defines an action of its Lie algebra on $M$.
\item For a Lie algebra $\g$ acting on $M$, an element $\alpha \in \g$ and a point $x \in M$, we denote by $\alpha(x) \in T_xM$ the value at $x$ of the vector field corresponding to
$\alpha$. We denote by $\g x \subset T_xM$ or by $\g (x) \subset
T_xM$ the image of the map $\alpha \mapsto \alpha(x)$ and by $\g_x
\subset \g$ its kernel.

\end{enumerate}
\subsubsection{\RamiAM{Differential forms}}

$ $
\RamiQAB{the numbering shift is manual. Should be checked in the end}

\begin{enumerate}
 \setcounter{enumi}{18}
\item For a top differential form \Rami{$\omega$ on \RamiAM{a} 
  manifold $M$},
 we define \Rami{its} absolute value \Rami{$|\omega|$ to be the corresponding measure on $M$} \RamiAM{(or on $M(F)$ in the algebraic case).}

\end{enumerate}
\subsection{ {Algebraic geometry}} \label{ssec:ag}
\subsubsection{Resolution of singularities} \label{sssec:rs}
$ $
In this paper we will need Hironaka's theory of \RamiM{resolution} of singularities. This theory was established in \cite{Hir}. A more recent overview can be found in \cite{Hir_mod}. \\ Let us summarize here the results we need.
\begin{definition}
Let $X$ be an algebraic variety.
\begin{itemize}
\item A \DrinD{\emph{resolution of singularities}} of $X$ is a proper \RamiD{map $p:Y \to X$ such that $Y$ is smooth and $p$ \RamiE{is a} birational equivalence.}
\item \DrinD{A subvariety $D \subset X$ is said to be a \emph{normal crossings divisor}
(or \emph{NC divisor}) if }
for any $x\in D$ there exists an \et neighborhood $\phi: U \to X$ of $x$
\RamiD{and an \et map $ {\alpha}:U \to \A^n$ such that $\phi^{-1}(D)= {\alpha}^{-1}(D')$, where $D' \subset \A^n$}
is a union of coordinate hyperplanes.

 {
\item 
\DrinD{A subvariety $D \subset X$ is said to be a \emph{strict normal crossings divisor}
(or \emph{SNC divisor}) if }
for any $x\in D$ there exists \RamiM{a} Zariski neighborhood $U \subset X$ of $x$
and an \et map $ {\alpha}:U \to \A^n$ such that $D \cap U= {\alpha}^{-1}(D')$, where $D' \subset \A^n$
 is a union of coordinate hyperplanes.
}
\item \DrinD{We say that a resolution of singularities $p:Y \to X$ \emph{resolves} (resp. \emph{strictly resolves}) a closed subvariety $D \subset X$ if $p^{-1}(D)$ is an NC divisor (resp. an SNC divisor).}
\end{itemize}
\end{definition}







\begin{theorem}[Hironaka] \label{thm:Hir}
Let $X$ be an algebraic variety and
\DrinD{$U\subset X$ a dense nonsingular open subset. Then there exists a resolution of singularities
$p:\tilde X \to X$ that resolves $X-U$ such that the map $p^{-1}(U) \to \RamiM{U}$ is an isomorphism.}
\end{theorem}
 {
There is a standard procedure to resolve a normal crossings divisor further to a strict normal crossings divisor,
 see e.g. \cite{Jon}.
This gives the following corollary.

\begin{corollary}
\DrinD{In Theorem~\ref{thm:Hir} one can replace ``resolves" by ``strictly resolves".}
%
\end{corollary}
}
\subsubsection{ {Nagata's  compactification theorem}} \label{sssec:nag}
 {
We will need the following theorem:

\begin{theorem}[Nagata (see e.g. \cite{Con_Nag})] \label{thm:Nag}
Let $\phi: X \to Y$ be a morphism of algebraic \RamiM{varieties}.
Then there exists a factorization $\phi=\phi' \circ i: X \to X' \to  Y$ such that  $i: X \to X'$ \RamiM{is} an  open
 embedding and $\phi':X' \to Y$ is proper.
\end{theorem}
}
\subsection{Distributions \RamiAE{in the non-Archimedean case}}\label{ssec:dist}
$ $\\
\RamiAE{We recall here the facts that we need about distributions in the non-Archimedean case. The Archimedean case will be discussed in \S\S\ref{ssec:arc_dist}. }

We will use the language of $l$-spaces and distributions on them. For an overview of this theory we refer
 the reader to \cite{BZ}.

\RamiAM{Let us briefly recall the basic notations and constructions of this theory;  all the notations except numbers (\ref{it:dist_notstd1}), (\ref{it:dist_notstd2}), (\ref{it:dist_notstd3}) are standard.}


\subsubsection{\RamiAM{Functional spaces}}\label{sssec:fun_sp}
Let $X$ be an $l$-space,{ i.e. a locally compact totally disconnected topological space. }
\begin{enumerate}
\item Denote by $C^\infty(X)$ the space
of smooth functions on $X$ (i.e. locally constant \RamiE{complex valued} functions).

\item Denote by $\Sc(X)$ the space
of \RamiK{\emph{Schwartz functions}} on $X$,
\RamiE{i.e. smooth, compactly
supported functions.} {Define the space of distributions} $\Sc^* (X):= \Sc(X)^*$ to be the dual
space to $\Sc(X)$\RamiT{, endowed with the weak dual space topology}. \RamiE{We will also denote by $C(X)$ the space of complex valued, continuous functions on $X$.}

\RamiE{\item By ``locally constant sheaf'' over an $l$-space we mean a locally constant sheaf of
\DrinA{finite-dimensional vector spaces over $\C$. }
In fact, we will
need only locally constant sheaves of rank $1$.}

\item For any locally constant sheaf $E$ over $X$, we denote by
$\Sc(X,E)$ the space of compactly supported sections of $E$, by $\Sc^*(X,E)$ its dual space,
 and by $C^\infty(X,E)$ the space of sections of $E$.
\RamiE{
We will also use the notation 
\DrinA{
$$C(X,E):= C^\infty(X,E) \otimes_{C^\infty(X)} C(X)$$
}for the space of
 continuous sections of $E$.}




\item
 {Let $\Sc_c^*(X,E)$  be the space of compactly supported distributions.
}
Note that we have a canonical embedding
 {$\Sc_c^*(X,E) \hookrightarrow (C^\infty(X,E))^*$.}
\item\label{it:dist_notstd1}
Suppose $X$ is an analytic variety. Then we define $D_X$ to be the sheaf of locally constant measures on $X$
 (i.e. measures that are locally isomorphic to the Haar measure on $F^n$).
 We set $\G(X):=\Sc^*(X,D_X)$ \RamiE{to be the space of generalized functions} and $\G(X,E):=\Sc^*(X,D_X \otimes
E^*)$ \RamiE{to be the space of generalized sections of $E$.}
\RamiE{Similarly, we define  $\G_c(X)$ and $\G_c(X,E)$. Note that we have natural embeddings $C^\infty(X,E) \subset C(X,E) \subset \G(X,E)$ and $\Sc(X,E) \subset \G_c(X,E)$ \RamiAQ{defined using the pairing between $E$ and $E^*$ followed by integration}.
We will identify these
spaces with their images and we will refer to the generalized  \RamiM{sections which} lie in $C^\infty(X,E)$  as ``smooth'' and those
which lie in $C(X,E)$ as ``continuous''}.
\end{enumerate}
\subsubsection{\RamiAM{Pullback and pushforward}}\label{sssec:op_dist}
\RamiAM{Let $\phi:X \to Y$ a \RamiE{continuous} map of $l$-spaces.}
\RamiQAB{the numbering shift is manual. Should be checked in the end}
\begin{enumerate}
 \setcounter{enumi}{6}
\item 
\RamiAM{W}e define $\phi^*:\RamiE{C^\infty(Y) \to C^\infty(X)}$
 to be the pullback
and $\phi_*:=(\phi^*)^* {|_{\Sc_c^*(X,E)}}:\RamiE{\Sc_c^*}(X) \to \RamiE{\Sc_c^*}(Y)$ to be the pushforward.
Similarly, we define \RamiE{$\phi^*:C^\infty(Y,E) \to C^\infty(X,\RamiM{\phi^*(E)})$ and
$\phi_*:\Sc_c^*(X,\DrinD{\phi}^*(E)) \to \Sc_c^*(Y,E)$}
for any locally constant sheaf
$E$.

\item \RamiAN{Assume that $\phi$ is proper. This allows us to extend the    pushforward
 to a map $\phi_*:\Sc^*(X) \to \Sc^*(Y)$ in the following way.} 
{\RamiAN{N}ote that $\phi^*(\Sc(Y)) \subset \Sc(X)$
and consider $\phi^*|_{\Sc(Y)}$
 as a map from $\Sc(Y)$ to $\Sc(X)$. So we can} \DrinA{define the pushforward
 $\phi_*:=(\phi^*|_{\Sc(Y)})^*:\Sc^*(X) \to \Sc^*(Y)$ extending the above map
  $\phi_*:\Sc_c^*(X) \to \Sc_c^*(Y)$.}
 Similarly, we define
$\phi_*:\Sc^*(X,\DrinD{\phi}^*(E)) \to \Sc^*(Y,E)$
for any locally constant sheaf
$E$.

 {\item
We can generalize the above two definitions in the following way.
Let
$\xi \in \Sc^*(X)$. \RamiM{Assume  $\phi|_{\supp(\xi)}$ is proper}.
\DrinD{Then $\phi^*(f)\cdot\xi$ has compact support for any $f\in\Sc (Y)$, so we can define
$\phi_*(\xi)\in\Sc^*(Y)$ by
$$\langle\phi_*(\xi) , f\rangle=\int\limits_X\phi^*(f)\cdot\xi:=\langle\phi^*(f)\cdot\xi,1\rangle,
\quad f\in\Sc (Y).$$
Similarly, for any locally constant sheaf $E$ on $Y$
one defines $\phi_*(\xi)$ if $\xi \in \Sc^*(X,\phi^*(E))$ is such that
$\phi|_{\supp(\xi)}$ is proper.}
}
\item \RamiE{Let $\phi:X \to Y$ be an analytic submersion of analytic manifolds. \RamiAN{Let us extend the pullback $\phi^*:C(Y)\to C(X)$ to a map $\phi^*:\G(Y) \to \G(X)$ in the following way. N}ote that \RamiAQ{since locally $\phi$ looks like a linear projection, we have}
 $$\phi_*(\Sc(X,D_X)) \subset \Sc(Y,D_Y).$$\RamiAQ{C}onsider $\phi_*|_{\Sc(X,D_X)}$
 as a map from $\Sc(X,D_X)$ to $\Sc(Y,D_Y)$.
 \DrinA{The pullback $\phi^*:=(\phi_*|_{\Sc(X,D_X)})^*:\G(Y) \to \G(X)$ extends the map $\phi^*:C(Y)\to C(X)$.}
}
\item \label{it:dist_notstd2}
For an analytic {submersion} $\phi:X \to Y$ of analytic manifolds\RamiM{,}
we define \RamiAQ{the line bundle of relative densities (i.e. the natural line bundle whose restriction to any fiber is the bundle of densities on it) by} $$D_Y^X:=\phi^*(D_Y^*)\otimes D_X.$$

For a locally constant sheaf $E$ over $Y$\RamiM{,} we denote  $\phi^{\RamiAN{!}}(E):=\phi^*(E)\otimes D_Y^X$.
\RamiQ{(Do you know any standard notation for this notion?)}
\RamiE{As before, we have the pushforward $\phi_*|_{\Sc(X,\phi^{\RamiAN{!}}(E))}:\Sc(X,\phi^{\RamiAN{!}}(E)) \to \Sc(Y,E)$
and the pullback $\phi^*:=(\phi_*|_{\Sc(X,\phi^{\RamiAN{!}}(E))})^*:\G(Y,E) \to \G(X,\phi^*(E))$.
}

\item \label{it:dist_notstd3} Let $T:X \to Y$ be an isomorphism of analytic manifolds. Note that $T_*=(T^{-1})^*$ both for
functions and for distributions. In this case, we will use the notation $T$ for both of these maps.

\end{enumerate}

%
%
%



\subsubsection{Fourier transform} \label{sssec:ft}

\begin{definition}\label{def:FT}
$ $
\begin{itemize}
\item Let $W$ be an $F$-vector space. We define the 
\term{Fourier transform}{FT}
$$\cF^{}:\Sc(W,D_W) \to \Sc(W^*)$$ by $$\cF(f)(\phi)=\int f \cdot(\psi \circ \phi), \quad \phi\in W^*.$$
We also define $$\cF^{*}:\Sc^*(W) \to \G(W^*),$$
 to be the dual map (when $W$ is replaced with $W^*$)
\item Let $X$ be an \RamiF{analytic manifold}. Similarly, we have  the partial Fourier transform
 $$\cF_{W}:\Sc(X \times W,D_X^{X \times W}) \to \Sc(X \times W^*)$$
 \RamiAQ{defined by $$\cF_{W}(f)|_{\{x\}\times W^{*}}:=\cF(f|_{\{x\}\times W})$$}and the dual map
 \RamiQ{(Should I add more details here?)}
$$\cF_{W}^{{*}}:\Sc^*(X \times W) \to \RamiF{\Sc^*(X \times W^*,D_X^{X \times W^*})}
  \DrinD{=\G(X\times W^*,D^{X\times W^*}_{W^*})}.$$
\end{itemize}

\end{definition}

We formulate here some standard properties of the Fourier transform \RamiM{which} we will use in the paper. 
\begin{proposition} \label{prop:F_prop}
Let $W$ \RamiAA{and $W'$} be  $F$-vector space\RamiAA{s} and $X$ be an \RamiF{analytic manifold}.
Let $\xi \in \Sc^*(X \times W)$.
\begin{enumerate}
\item Let $U \subset X$ \RamiT{be an open set.} Then $\cF_{W}^{{*}}(\xi)|_{\RamiT{U\times W^{\RamiAQ{*}}}}=\cF_{W}^{{*}}(\xi|_{\RamiT{U\times W}})$.
\item Let $f\in C^\infty(X)$ be a locally constant function. Then $\cF_{W}^{{*}}(f\xi)=f\cF_{W}^{{*}}(\xi).$
\item\label{prop:F_prop:3} Let $p:X \to Y$ be a proper map of $l$-spaces. Then $\cF_{W}^{{*}}(p_*\xi)=p_*\cF_{W}^{{*}}(\xi).$
\RamiAA{\item \label{prop:F_prop:4}
Let $\eta \in \Sc^*(X \times W\times W')$. Then $\cF_{ W\times W'}(\eta)=\cF_{W}(\cF_{W'}(\eta)).$}
\end{enumerate}
\end{proposition}

In order to formulate the last \RamiF{properties}, we will need the following \RamiF{notation}.

\begin{notation}\label{not:F_prop}
\RamiF{
Let $W,L$ be $F$-vector spaces and $X$ be an analytic manifold. Let $\nu:X \to Hom(L,W)$ be a continuous map. Then
\begin{enumerate}
\item  $\nu^t:X \to Hom(W^*,L^*)$ \DrinH{denotes} the map given by
$\nu^t(x)=\nu(x)^t$;

\item   $\rho_\nu:X \times L \to X \times W$  \DrinH{denotes} the map given by $\rho_\nu(x,y)=(x,\nu(x)(y))$;
\DrinH{in particular, we use this notation when $L=W$ and $\nu:X \to F\subset \End(W)$ is a scalar function;}

\item  $Mon(L,W) \subset Hom(L,W)$  \DrinH{denotes} the space of linear embeddings from $L$ to $W$.
\end{enumerate}
}

\end{notation}

\begin{proposition} \label{prop:F_prop2}

Let $W,L$ be $F$-vector spaces and $X$ be an analytic manifold. Let $\nu:X \to Mon(L,W)$ be a continuous map.
Then
\RamiG{
the following diagrams are commutative:

\begin{equation}
\xymatrix
{        \Sc(X \times W^*,D_X^{X \times W^*})  \ar@{->}^{\quad \quad \cF_{W^*}}[r] \ar@{->}^{(\rho_{{\nu^t}})_*}[d]& \Sc(X \times W) \ar@{->}^{(\rho_{\nu})^* }[d]  \\
         \Sc(X \times L^*,D_X^{X \times L^*})  \ar@{->}^{\quad \quad \cF_{L^*}}[r]                                 & \Sc(X \times L)   \\
}
\end{equation}

\begin{equation}
\xymatrix
{        \Sc^*(X \times W^*,D_X^{X \times W^*})  \ar@{<-}^{\quad \quad \cF^*_{W}}[r] \ar@{<-}^{(\rho_{{\nu^t}})^*}[d]& \Sc^*(X \times W) \ar@{<-}^{(\rho_{\nu})_* }[d]  \\
         \Sc^*(X \times L^*,D_X^{X \times L^*})  \ar@{<-}^{\quad \quad \cF^*_L}[r]                                 & \Sc^*(X \times L)   \\
}
\end{equation}
}
\end{proposition}
Note that since $\rho_\nu$ is an embedding, and $\rho_{\nu^t}$ is a submersion,
the inverse and the direct images in the diagrams
are defined.



\subsubsection{The wave front set} \label{sssec:wfs}
{As it was mentioned earlier, we will prove a stronger version of Theorem \ref{thm:main_a}, which has to
do with the wave front set. The  wave front set is an important invariant of a distribution $\xi$ on an analytic manifold $X$, which was introduced  in  \cite{Hor} in the Archimedean case and then adapted in \cite{Hef} to the non-Archimedean case.

The wave front  set is a closed subset of $T^*X$. We will denote it by $\WF(\xi)$.
The definition of $\WF(\xi)$ will be recalled in Appendix \ref{app:WF}.
Here we list the properties of the wave front set that will be used in this paper.
 Most of them are\RamiAE{ adaptations of results from \cite{Hor}. Some are} proved in \cite{Hef} and \cite{Aiz}, the rest will be proved in Appendix \ref{app:WF}.


\begin{proposition} \label{prop:WF_prop}
$ $
Let $X$ be an analytic variety and $E$ a \RamiF{locally
constant sheaf} over it. Let $\xi \in \G(X,E)$. Then we have:
\begin{enumerate}
\item \label{prop:WF_prop:1} $P_{T^*(X)}(WF(\xi))=WF(\xi) \cap X =\Supp(\xi),$ \RamiAQ{where $\Supp(\xi)$ denotes the usual support of $\xi$.} \RamiE{Here we identify $X$ with the zero section inside $T^*{X}$ \RamiAM{and $P_{T^*(X)}:T^*{X}\to X$ is the projection}.}
\item \label{prop:WF_prop:2} $WF(\xi) \subset X$ if and only if $\xi$ \RamiE{is smooth}.
\item \label{prop:WF_prop:3} Let $U \subset X$ be an open set. Then $WF(\xi|_U) = WF(\xi) \cap T^*(U).$
\item  \label{prop:WF_prop:4} Let $\xi' \in \G(X,E)$ and $f,f' \subset C^\infty(X).$  Then
$$WF(f \xi+f' \xi') \subset WF(\xi) \cup
WF(\xi').$$
\item \label{it:G_act} Let $G$ be an analytic group acting on $X$ and $E$. Suppose $\xi$ is
$G$-{invariant.}
Then $$ WF(\xi) \subset \{(x,v) \in T^*X(F)|v(\g x)=0\}=\bigcup_{x \in X} CN_{Gx}^X,$$
where $\g$ is the Lie algebra of $G$.
\end{enumerate}

\end{proposition}
In order to formulate the rest of the properties we will need the following notions:
\begin{definition}\label{def:dir_and_inv_image}
Let $X$ be an analytic variety. Let $A \subset T^*(X).$
\begin{enumerate}
\item \label{def:dir_and_inv_image:1}
  \RamiCP{We say that $A$ is \RamiK{\emph{conic}} if it is stable with respect to the homothety action
of $F^\times$ on $T^* X$, given by $\rho_\lambda(x,v)=(x,\lambda v)$.}

\item \label{def:dir_and_inv_image:2}
If $p: Y \to X$ is an analytic map we define $p^*(A)\subset T^*(Y)$ by
$$p^*(A):=\{(y,v) \in T^*(Y)|\exists w \in (dp^*)^{-1}(v)\subset T^{*}_{p(y)}\RamiAS{X} \text{ with } (p(y),w) \in A \}.$$
\item \label{def:dir_and_inv_image:3}
 If  $p:  X \to Y $ is an analytic map we define $p_*(A)\subset T^*(Y)$ by
 $$p_*(A):= \{(y,v) \in T^*(Y)|\exists x \in p^{-1}(y) \text{ with } (x,(d_xp)^*(v)) \in A \}.$$

\end{enumerate}
\end{definition}
\RamiAE{
\begin{rem}\label{rem:simp_gem}
We can describe the procedures of direct and inverse images in terms of symplectic geometry.

Namely, let $\pi: M \to N$ be a map of manifolds.
It gives rise to a correspondence $\Lambda_\pi \subset T^*(M) \times T^*(N)$ by $\Lambda_\pi=\{((x,v),(y,w))|y=\pi(x), v=d\pi^*(w) \}$.

Now let $S$ and $T$ be
 a symplectic manifold and $\Lambda \subset S \times T$ be a  correspondence.
 For a subset $Z \subset S$, we set $\Lambda(Z)=\{y\in T|\exists x\in Z \text{ such that }(x,y)\in \Lambda\}$.

This gives the following alternative definition for direct and inverse images: \begin{itemize}
\item for a subset $Z \subset T^*(M)$, we have $\pi_*(Z)=\Lambda_\pi(Z)$.
\item for a subset $Z \subset T^*(N)$, we have $\pi^*(Z)=\Lambda^{-1}_\pi(Z)$. Here $\Lambda^{-1}_\pi$ is $\Lambda_\pi$ considered as a subset  of $T^*(N) \times T^*(M).$
\end{itemize}

\end{rem}
}
%
\begin{proposition} \label{prop:WF_prop2}
$ $
Let $X$ be an analytic variety. Then we have:
\begin{enumerate}
\item $WF(\xi)$ is conic.
\item \label{prop:WF_prop2:2} Let \RamiL{$E$ be a locally
constant sheaf over $X$,  let $\xi \in \G(X,E)$ and let}
$p: Y \to X$ be an analytic submersion. Then $WF(p^*(\xi)) \subset p^*(WF(\xi))$.
\item \label{prop:WF_prop2:3} Let $ {q}: X \to Y$ be an analytic map\RamiL{, let $E$ be a locally
constant sheaf over $Y$  and let $\xi \in \Sc(X,q^*\RamiM{(E)})$.}
\RamiM{Assume $q|_{\supp (\xi)}$ is proper.}
Then $WF( {q}_*(\xi)) \subset  {q}_*(WF(\xi))$.
\end{enumerate}
\end{proposition}
}
\RamiT{In \S\ref{sec:pf_futer} we will need the following more complicated properties of the wave front set:
\RamiAB{
\begin{notation}
Let $X$ be an analytic manifold.
For a closed conic set $\Gamma \subset T^*X$ we denote by $\G_\Gamma(X)$ the space of generalized functions whose wave front set is in $\Gamma$. We will consider this space equipped with its natural topology which we describe in Appendix \ref{app:WF}.  We will use similar notations for other types of generalized sections.
\end{notation}
}

\begin{proposition} \label{prop:WF_prop3}
We have the following generalization of Proposition \ref{prop:WF_prop2} \eqref{prop:WF_prop2:2}.
Let $p: Y \to X$ be an analytic map of analytic manifolds, let
$$N_p= \{(x,v)\in T^{*}X| x = p(y)  \text{ and } d_{\RamiAB{y}}^* p(v)=0 \text{ for some } y\in Y \}.$$
Let $E$ be a locally
constant sheaf over $X$.
\RamiAB{
Let $\Gamma \subset T^{*}X$ be a conic closed subset such that $\Gamma \cap N_p \subset X$.
}

Then the  map $p^{*}:C^\infty(X,E)\to C^\infty(Y,p^*(E))$
 has a unique continuous extension
to a
map
$p^{*}:\G_\RamiAB{\Gamma}(X,E)\to \G(Y,p^*(E))$.
 Moreover
for any $\xi \in \G_\RamiAB{\Gamma}(X,E)$ we have
$WF(p^*(\xi)) \subset p^*(WF(\xi))$.
\end{proposition}
}
\RamiAD{
\begin{definition}\label{def:cont_dep}
Let $\xi\in \G(X \times Y)$ be a generalized function on a product of analytic manifolds. We will say that $\xi$ depends continuously on $Y$ if
for any $f \in \Sc(X,D_X)$ the generalized function $\xi_f \in \G(Y) $ given by $\xi_f(g)=\xi(f \boxtimes g)$ is continuous.\footnote{\RamiAQ{$f \boxtimes g \in \Sc(X\times Y,D_X \otimes D_Y)\cong \Sc(X\times Y,D_{X \times Y})$ denotes the density given by $(f \boxtimes g)(x,y):=f(x) \otimes g(y)$.}}
In this case we \DrinI{define} $\xi|_{X \times \{y\}} \in  \G(X \times \{y\}) $   by  $\xi|_{X \times \{y\}}(f):=\xi_f(y)$.
\end{definition}

{
\begin{rem}
In the situation of Definition~\ref{def:cont_dep} the generalized functions $\xi_y:=\xi|_{X \times \{y\}}$, $y\in Y$, form a \emph{continuous family} (i.e., the map $Y\to\G(X)$ defined by $y\mapsto\xi_y$ is continuous). Thus one  gets a bijection between generalized functions on $X \times Y$ depending continuously on $Y$ and continuous families of generalized functions on $X$ parametrized by~$Y$.
\end{rem}
}

\begin{proposition}   \label{prop:Treves}
Let $\xi\in \G(X \times Y)$ be a generalized function on a product of analytic manifolds. Assume that
{
\begin{equation} \label{e:wavefront_condit}
WF(\xi) \cap CN_{X \times \{y\}}^{X \times Y}\subset X \times Y
\end{equation}
}
and that $\xi$ depends continuously\footnote{In fact, \eqref{e:wavefront_condit} implies that $\xi$ depends continuously on $Y$ (cf. the discussion after Proposition 6.11 in \cite{tre}).
We will not need this implication.} on $Y$.
Then
$\xi|_{X \times \{y\}}=j^*(\xi)$ where $j:X \times \{y\}\hookrightarrow X \times Y$ is the embedding.
\end{proposition}
\RamiAM{
}
}

Combining Proposition~\ref{prop:WF_prop3}  and Proposition~\ref{prop:Treves} we get the following
\begin{cor}   \label{c:Treves}
In the situation of Proposition~\ref{prop:Treves} one has
$$WF(\xi|_{X \times \{y\}})\subset j^*(WF(\xi)).$$
\end{cor}

\section{WF-holonomic distributions}\label{sec:WF-Holonomic}
\RamiQ{
$ $
It seems that we  need to weaken the notion of WF-holonomic to make it more appropriate to the analytic world. Although
 it should not matter for Theorem \ref{thm:main_b}, it is relevant for Theorem \ref{thm:main_c}.
Since we are proving Theorem \ref{thm:main_b} via Theorem \ref{thm:main_c}, it is  more convenient to use only one
 notion of WF-holonomic. The only disadvantage is  that we will prove regularity only  in an analytically open dense
 set and not a Zariski open  dense set. It should be fixed in the next version, when we give an explicit bound.}

  \DrinD{
 \subsection{Recollections on isotropic and Lagrangian conic subvarieties of $T^*(X)$}\label{sec:lag}
 Let $M$ be a symplectic algebraic manifold and $V\subset M$ a {constructible subset\footnote{{A  subset of an algebraic variety is said to be constructible if it is a finite union of locally closed subsets. A theorem of Chevalley says that the image of a constructible subset under a regular map is constructible. (A similar statement for preimages is obvious.)}}.}
 We say that $V$ is \emph{isotropic} (resp. \emph{Lagrangian}) if there is an open dense {subset  $V'\subset V$ which is a smooth isotropic (resp. Lagrangian) locally closed subvariety  in $M$}.


 \begin{rem}    \label{r:easy_facts}
 The closure of an  isotropic (resp. Lagrangian) subset is isotropic (resp. Lagrangian).
 The union of two  isotropic (resp. Lagrangian) subset is isotropic (resp. Lagrangian).
 \end{rem}

\begin{proposition} \label{prop:sub_iso}
If $V\subset M$ is isotropic then so is any {constructible subset} $Z\subset V$.
\end{proposition}

The statement is nontrivial because $Z$ may be contained in the set of singular points of $V$.
For a proof, see, e.g., \cite[Proposition 1.3.30]{CG} and  \cite[\S 1.5.16]{CG}.

Now let $M=T^*(X)$, where $X$ is a smooth algebraic manifold. The multiplicative group acts
on $M$ by homotheties. A subvariety of $M$ is said to be \emph{conic} if it is stable with respect to this action. If $A\subset X$ is a smooth algebraic subvariety then the conormal bundle
$CN_A^X$ and its closure $\overline{CN_{A}^X}$ are conic Lagrangian subvarieties of $T^*(X)$.
It is well known that any closed conic Lagrangian subvariety of $T^*(X)$ is a finite union of
varieties of the form $\overline{CN_{A}^X}$. Here is a slightly more general statement.

\begin{lemma}    \label{lem:isotr_discrip}
Let $X$ be an algebraic manifold and $C \subset T^*(X)$ a closed conic algebraic subvariety.
Then the following properties of $C$ are equivalent:

\begin{enumerate}
\item $C$ is isotropic; \label{1}
\item $C$ is contained in a Lagrangian subvariety
of $T^*(X)$; \label{2}
\item  There is a finite collection of smooth locally closed subvarieties $A_i \subset X$ such that
$$C \subset \bigcup_i \overline{CN_{A_i}^X}\,;$$ \label{3}
\item There is a finite collection of smooth locally closed subvarieties $A_i \subset X$ such that
$$C \subset \bigcup_i {CN_{A_i}^X}\,.$$ \label{4}
\end{enumerate}
\end{lemma}
This lemma is standard. For completeness, we include its proof in
\RamiAF{Appendix \ref{app:symp_geom}.} 


\medskip

Now let $S\subset T^*X(F)$ be \emph{any} conic subset (not necessarily an algebraic subvariety). Its Zariski closure $\bar S\subset T^*X$ is also conic.

\begin{lemma}   \label{l:closure}
$\bar S$ has the equivalent properties from Lemma~\ref{lem:isotr_discrip} if and only if
there is a finite collection of smooth locally closed subvarieties $A_i \subset X$ such that
$S \subset \bigcup\limits_i {CN_{A_i}^X (F)}\,$.
\end{lemma}

\begin{proof}
If $\bar S \subset \bigcup\limits_i  CN_{A_i}^X$\RamiM{,} then
$S \subset \bigcup\limits_i {CN_{A_i}^X (F)}\,$. If
$S \subset \bigcup\limits_i {CN_{A_i}^X (F)}\,$\RamiM{,} then
$\bar S \subset \bigcup\limits_i  \overline{CN_{A_i}^X}$.
\end{proof}
}

{The following lemma is well known (see Appendix \ref{app:symp_geom} for a proof).}
\RamiAF{
\begin{lemma}\label{lem:iso_dir_invs_im}
Let $p:X\to Y$ be a morphism of algebraic manifolds. {Let $T\subset T^*X$ and $S\subset T^*Y$ be constructible subsets.}
\begin{enumerate}
\item If $T$ is isotropic then $p_*(T)$ is.
\item If $S$ is isotropic then $p^*(S)$ is.
\end{enumerate}
\end{lemma}
}
{For the definition of $p_*$ and $p^*$, see Definition~\ref{def:dir_and_inv_image} and Remark~\ref{rem:simp_gem}. Note that since $T$ and $S$ are constructible so are  $p_*(T)$ and $p^*(S)$.}


 \subsection{WF-holonomic distributions}

\newWFHol{
\begin{definition}
Let $X$ be an \RamiJ{algebraic manifold over $F$} \DrinD{and let $E$ be a locally constant sheaf on $X(F)$}.
 A distribution $\xi \in \Sc^*(X\RamiJ{(F)},E)$ is said to be \emph{\RamiJ{algebraically} WF-holonomic} if \RamiAF{the \RamiAH{Zariski} closure of $WF(\xi)$ is isotropic.} 
\end{definition}
\begin{remark}
{
By Lemmas~\ref{lem:isotr_discrip} and \ref{l:closure}, $\xi$ is algebraically WF-holonomic if and only if
$WF(\xi) \subset \bigcup\limits_i {CN_{A_i}^X (F)}$ for some smooth locally closed subvarieties
$A_1,\ldots, A_n\subset X$.
}
\end{remark}


\begin{remark}
It can happen that the Zariski closure of $WF(\xi)$ is isotropic but not Lagrangian (simple examples are given in \cite[Appendix A]{Dri}).
\end{remark}

\RamiJ{
\begin{remark}
One can also define a more general notion of ``analytically WF-holonomic distribution'' for analytic manifolds. However, we
will not discuss it in this paper. So we will use the expression ``WF-holonomic'' as a shorthand for
 ``algebraically WF-holonomic''.
\end{remark}
}

}
\begin{remark}
{In general, the notion of WF-holonomicity is not as powerful as the notion  of holonomicity given by the theory
 of D-modules.
 For example, it is not true that the Fourier transform of a  WF-holonomic distribution on an affine space is WF-holonomic. Yet if the variety $X$ is compact\RamiM{,} then the notion of WF-holonomicity
 seems to be a good candidate for  replacing the notion of holonomicity in the non-Archimedean
case.
}
\RamiQ{Is this remark appropriate? }
\end{remark}

The next lemma {follows immediately from \DrinD{statements (2) and (3) of } Proposition~\ref{prop:WF_prop}}.
\begin{lemma}\label{lem:hol_reg}
Let $X$
be an \RamiJ{algebraic $F$-manifold} and $E$ a \RamiF{locally
constant sheaf} over \RamiJ{$X(F)$}.
If $\xi \in \G(X (F),E)$ is WF-holonomic
then there exists \RamiJ{a Zariski} open dense subset $U \subset X$
such that $\xi|_{\RamiJ{U(F)}}$
\RamiE{is smooth}.
\end{lemma}

\RamiAF{The fact that inverse and direct images preserve \RamiAH{isotropicity} (Lemma \ref{lem:iso_dir_invs_im}) and the properties of the wave front set (Propositions \ref{prop:WF_prop2} and \ref{prop:WF_prop3}) imply the following proposition:}

\RamiL{
\begin{proposition} \label{prop:WF_Hol}
$ $
Let $X$
be an \RamiJ{algebraic $F$-manifold}.
\begin{enumerate}
\item \label{prop:WF_Hol:1} Let $E$ be a locally
constant sheaf over $X(F)$,  let $\xi \in \G(X,E)$ be a WF-holonomic generalized section and let
$p: Y \to X$ be \RamiAF{a morphism.
Assume that $\WF(\xi) \cap N_p \subset X$. }
Then $p^*(\xi)$ is WF-holonomic.
\item Let $ {q}: X \to Y$ be a \RamiR{regular} map, let $E$ be a locally
constant sheaf over $Y(F)$  and let $\xi \in \Sc(X,q^*\RamiM{(E)})$
 be a WF-holonomic distribution.
Assume that the map $\supp(\xi)\to Y(F)$ induced by $q$ 
is proper (as a continuous map).
Then $ {q}_*(\xi)$
is WF-holonomic.
\end{enumerate}
\end{proposition}
}


\RamiG{
We will also use the following corollary of Proposition \ref{prop:WF_prop} \RamiAH{\eqref{it:G_act}}
\begin{corollary} \label{cor:WF_Hol}
$ $
Let $X$ be an \RamiJ{algebraic manifold} and $E$ a \RamiF{locally
constant sheaf} over \RamiJ{$X(F)$}. Let an \RamiJ{algebraic} group $G$ act on $X$ and \RamiJ{let $G(F)$ act on}
 $E$. Let $U \subset X$ be a
$G$-stable
open set and $Z=X-U$. Let $\xi \in \G(X\RamiJ{(F)},E)^{\RamiJ{G(F)}}$. Suppose $Z$ has a finite number of
$G$-orbits and $\xi|_{\RamiJ{U(F)}}$ is smooth. Then $\xi$ is WF-holonomic.
\end{corollary}
By twisting the action of $G$ on $E$ by a quasi-character, we obtain the following \RamiM{version} of
Corollary~\ref{cor:WF_Hol}.
\begin{corollary} \label{cor:WF_Hol2}
Let $X,G,E,Z,U$ be as in Corollary \ref{cor:WF_Hol}.
Let
\RamiAS{$\xi \in \G(X{(F)},E)^{}$}. \RamiM{Suppose $\xi|_{\RamiJ{U(F)}}$} is smooth and the line $\C\xi\subset\G(X\RamiJ{(F)},E)$
is $G\RamiJ{(F)}$-stable (i.e., $\xi \in \G(X\RamiJ{(F)},E)^{G\RamiJ{(F)},\chi}$ for some quasi-character
$\chi :G\RamiJ{(F)}\to\C^{\times}$\RamiJ{)}. Then
$\xi$ is WF-holonomic.
\end{corollary}
}
\section{ Proof of \DrinI{Theorems \ref{thm:main_a}-\ref{thm:main_bb}}}\label{sec:pf_main_bb}

\DrinD{In this section we prove the theorems formulated in \S\ref{ss:main}.


Using the notion of \RamiJ{WF-holonomic} distribution from \S\ref{sec:WF-Holonomic}, one can reformulate Theorem~\ref{thm:main_bb} as follows.

\begin{theorem}   \label{t:C}
Let $W$ be a finite-dimensional $F$-vector space and $X,Y$ be algebraic manifolds. Let
$\phi: X \to Y \times W$ be a proper map
 and let $\omega$ be a \RamiR{regular} top differential form on $X$.
Then the partial  Fourier transform
\footnote{{Partial Fourier transform was introduced in
Definition \ref{def:FT}. The symbols $\Sc^*$ and $\G$ were introduced in \S\S\S\ref{sssec:fun_sp}. For the symbols
$D_{Y(F)}^{Y(F)\times W^*}$ and $D^{Y(F)\times W^*}_{W^*}$, see  \S\S\S\ref{sssec:op_dist}.}}
$$\cF_W^*(\phi_*(|\omega|)) \in \Sc^*( {Y(F)} \times W^*,D_{Y(F)}^{Y(F)\times W^*})=\G( {Y(F)}\times
 W^*,D^{Y(F)\times W^*}_{W^*})$$  is
WF-holonomic.



%
\end{theorem}

Theorem~\ref{thm:main_b} is a particular case of Theorem~\ref{t:C} when
$Y$ is a point. Theorem \ref{thm:main_a} follows from Theorem~\ref{thm:main_b}
by virtue of Lemma \ref{lem:hol_reg}. Thus it remains to prove Theorem~\ref{t:C}.
}

\subsection{Reduction to the key special case}  \label{ss:Reduction to key}
$ $
\RamiAM{
\begin{notation}
For a vector space $W,$ we denote by $\overline W$ the projective space of one dimensional subspaces of $W \oplus F$. We consider $W$ as an open subset of $\overline W$.
\end{notation}}
\RamiR{Using Hironaka's \DrinG{theorem and Nagata's theorem}, we}
 will deduce Theorem \ref{t:C} from \DrinG{Proposition~\ref{thm:main_c}, which is, in fact, a special case of  Theorem \ref{t:C}. To formulate Proposition~\ref{thm:main_c}, we need some notation.}

\DrinB{
\begin{notation}\label{not:eta}
Let $Y$ be an  algebraic
manifold and $W$ a vector space. Let $\phi:Y\to \overline{W}$ be an
 {algebraic}
map.
 We set $Y_0 := \phi^{-1}(W)$  and $Y_{\infty}:=\phi^{-1}(W_\infty)$. Assume $Y_0$  is
dense {in $Y$}. Let $\omega$ be a  {rational}
top differential form on $Y$ which is
  {regular}
 on $Y_0$ and let
 $\omega_0:=\omega|_{Y_0}$.  Define $i: Y_0 \hookrightarrow Y\times W$ by $i(y):=(y,\phi (y))$ and
set
 {$$\eta_{\phi,\omega}:=\DrinD{i}_*(|\omega_0|)\in\Sc^*(Y(F)\times W)\, .$$
We also set $$\hat{\eta}_{\phi,\omega}:= \cF^{*}_{W}(\eta_{\phi,\omega})\in
\DrinD{\Sc^*(Y(F) \times W^*,D_{Y(F)}^{Y(F)\times W^*})=
\G(Y\RamiM{(F)} \times W^*,D^{Y(F)\times W^*}_{W^*})}.$$}
\end{notation}

\begin{rem}\label{rem:eta}
$i_*(|\omega_0|)$ is a well-defined measure because the embedding $i$ is closed; to see this,
represent $i: Y_0 \hookrightarrow Y\times W$ as the composition
\DrinC{$$Y_0 {\buildrel{\sim}\over{\longleftarrow}}\Gamma_{\phi}\cap (Y_0\times W)=
\Gamma_{\phi}\cap (Y\times W)\hookrightarrow Y\times W,$$}
where $\Gamma_{\phi}\subset Y\times\overline{W}$ is the graph of $\phi$.
\end{rem}
}

\RamiR{
Now we can formulate the key special case of Theorem \ref{t:C}:}
\begin{proposition} \label{thm:main_c}
Let $Y, W, \phi,\omega, Y_\infty,Y_0$ be as in Notation \ref{not:eta}. Let $Z \subset Y$ be the zero locus of
$\omega$. \RamiM{Assume  $Z \cup  Y_\infty$} is an SNC \RamiP{divisor.}
Then the partial Fourier transform
$\RamiM{\hat \eta_{\phi,\omega}}$
is \RamiJ{WF-holonomic}.
\end{proposition}

We will prove this proposition in section \ref{sec:pf_C}.
\RamiR{
\begin{remark}
Note that Proposition \ref{thm:main_c} is indeed a special case of Theorem \ref{t:C}. Namely, if we take $X,Y,\phi:X \to Y \times W$ and $\omega$
 from Theorem \ref{t:C} equal to $Y_0,Y,i:Y_0 \to Y \times W$ and $\omega_0$ from Notation \ref{not:eta},
 then we obtain the assertion of Proposition~\ref{thm:main_c}.
\DrinG{Taking $\phi$ to be equal to $i$} is possible because, as mentioned in Remark \ref{rem:eta}, the map $i$ is a closed embedding and hence proper.
\end{remark}
}


\RamiF{In some cases, one can describe
$ {\hat \eta_{\phi,\omega}}$
explicitly. Namely, we have the following
straightforward calculation:}
\begin{lemma} \label{lem:exp_comp}
Let $(Y,W,\phi,\omega)$ be as above. Suppose $\Im \phi \subset W$.
Define  $f_\phi \in C^{\infty}(Y\RamiM{(F)}\times W^*)$ by $f_{\phi}(y,\xi):=\psi(\langle \xi,\phi(y) \rangle).$

Then the generalized section $ {\hat \eta_{\phi,\omega}}$ \DrinD{is equal}
 {to} the continuous section
$f_{\phi} \cdot  {pr_{\RamiS{W}}}^{*}(|\omega|)$, where $pr_{\RamiS{W}}:Y\times W \to Y$ is the projection.
\end{lemma}
\DrinC{
\begin{rem} \label{rem:exp_comp}
In fact, the formula
\begin{equation}   \label{e:always_true'}
 {\hat \eta_{\phi,\omega}}=f_{\phi} \cdot  {pr}^{*}(|\omega|)
\end{equation}
holds without assuming that $\Im \phi \subset W$ if the r.h.s. of \eqref{e:always_true'} is understood appropriately.
 More precisely, Definition~\ref{def:FT} and the definition of $\eta_{\phi,\omega}$
(see Notation~\ref{not:eta})
immediately imply that the scalar product of $ {\hat \eta_{\phi,\omega}}$ with any
$h\in  \Sc (Y\DrinD{(F)}\times W^*,D^{Y\RamiM{(F)}\times W^*}_{Y\RamiM{(F)}} )$ equals
the
\RamiR{iterated}
integral $$\int\limits_{Y_0}\int\limits_\RamiW{W^*} hf_{\phi} \cdot  {pr}^{*}(|\omega|)$$
(the latter makes sense because after integrating along $\RamiW{W^*}$\RamiM{,} one gets a measure on $Y_0\RamiM{(F)}$
\emph{with compact support}).
\end{rem}
}

Now let us prove Theorem  {\ref{t:C}} using   {Proposition} \ref{thm:main_c}.

 {
\begin{proof}[Proof of Theorem  {\ref{t:C}}]


  \DrinG{Applying Nagata's Theorem~\ref{thm:Nag} to the composition
$X\overset{\phi}\longrightarrow Y\times W\hookrightarrow  Y\times \overline W$ we get a commutative diagram
\[
\xymatrix
{        X   \ar@{}[r]|-*[@]{\hookrightarrow}   \ar@{->}_{\phi}[d]        & \overline X \ar@{->}_{\overline \phi}[d] \\
         Y\times W \    \ar@{}[r]|-*[@]{\hookrightarrow}   &  Y\times\overline W  \\ }
\]
in which the map $\overline X\to Y\times \overline W$ is proper and the map $X\RamiR{\hookrightarrow} \overline X$ is an open embedding.
Identify  $X$ with its image in $\overline X$.}

%

Let $Z \subset X$ be the zero locus of $\omega$, $\overline Z$ be its closure in $\overline X$ and
$\overline {X}_\infty:=\overline {X}-X$.
Let $\Xi :=\overline Z \cup \overline {X}_\infty$ and $U:=\overline X-\Xi\DrinH{\subset X}$. 
Let $\RamiS{\rho}:\tilde X \to \overline{X}$  be a resolution of singularities of $\overline{X}$ that strictly resolves $\Xi$\RamiM{,}
such that
$\RamiS{\rho}_{\RamiS{\rho}^{-1}(U\RamiI{)}}:\RamiS{\rho}^{-1}(U) \to U$
is an isomorphism.
Identify $U$ with $\RamiS{\rho}^{-1}({U})$.

Let $\pi_Y:Y \times \overline W \to Y$ and $\pi_{\overline W}:Y \times \overline W \to \overline W$
 be the projection\DrinH{s}.
\RamiR{Let $\alpha_0:\tilde X \to Y$ be the composition
$$\tilde X \overset{\RamiS{\rho}}{\to}  \DrinH{\overline X \overset{\overline\phi}{\to} Y} \times
 \overline W \overset{\pi_{ Y}}{\to} Y$$
and $\alpha:\tilde X \times W \to Y \times W$ be $\alpha_0 \times Id_{\DrinH{W}}$.
 }
Clearly $\alpha$ is proper. Let $\beta: \tilde X \to \overline W$ be
\RamiR{the composition
$$\tilde X \overset{\RamiS{\rho}}{\to}  \DrinH{\overline X \overset{\overline\phi}{\to} Y} \times  \overline W \overset{\pi_{\overline W}}{\to} \overline W$$}
Let $\omega'=\omega|_U$\RamiM{,} and consider $\omega'$ as a rational form on $\tilde X$. Let $Z' \subset \tilde X$
 be its zero locus.
Note that
$Z' \cup \beta^{-1}(W_\infty)=\RamiS{\rho}^{-1}(\Xi)$ is an SNC divisor.
\RamiAQ{Let $j:U \to \tilde X$ be the open dense embedding and $i:U \to \tilde X \times W$ be the map given by $i(x):=(x,\beta (y))$}.
We have:
$$\phi_*(|\omega|)=\RamiAQ{\phi_* (j_*(|\omega'|))=(\phi \circ j)_*(|\omega'|)=(\alpha \circ i)_*(|\omega'|)=\alpha_*(i_*(|\omega'|))=}\alpha_*(\eta_{\beta,\omega'})$$
and hence\RamiAQ{, by a standard property of Fourier transform (Proposition \ref{prop:F_prop}\eqref{prop:F_prop:3}), we get}
\RamiAQ{$$\cF^*_{W}(\phi_*(|\omega|))=\alpha_*(\hat \eta_{\beta,\omega'}).$$}
By  {Proposition} \ref{thm:main_c}, the distribution
$\hat \eta_{\beta,\omega\RamiAQ{'}}$
 is  WF-holonomic. Thus by Proposition  \ref{prop:WF_Hol}
$\cF^*_{W}(\phi_*(|\omega|))$
is WF-holonomic\RamiM{.}
\end{proof}
}

\subsection{Proof of \DrinH{ Proposition \ref{thm:main_c}} } \label{sec:pf_C}
\setcounter{lemma}{0}
The proof of  {Proposition}   \ref{thm:main_c} is based on
the key Lemmas~\ref{lem:red_to_dim_1} and \ref{lem:dim_1} below.
\subsubsection{\DrinB{\RamiR{T}he key lemmas}}
\DrinB{
Recall that if $W$ is a finite-dimensional vector space over $F$\RamiM{,} then  $\overline{W}$ stands for the space of lines in $\RamiG{ W \oplus F}$. The image in $\overline{W}$ of a nonzero vector
$(w,a)\in W\oplus F$ will be denoted by $\RamiG{(w:a)}$.

\begin{lemma} \label{lem:red_to_dim_1}
Let $Y$ be an  {algebraic} manifold and $W$  a vector space over $F$,  {with}  $\dim W<\infty$.
Let $\phi :Y\to \overline{W}$ be a map defined by $\phi(y)=\RamiG{( {\alpha}(y):p(y))}$,
 where $ {\alpha}:Y\to W$ and
$p:Y\to F$ are \RamiJ{regular}\RamiF{, $p \neq 0$ on a dense subset} and $ {\alpha}$ has no zeros. Let $\omega$ and
$\eta_{\phi,\omega}\in\Sc^*(Y {(F)}\times W)$ be as in Notation \ref{not:eta},
so $  {\hat \eta}_{\phi,\omega}\in
 \RamiF{\Sc^*}(Y {(F)}\times W^*,D^{Y {(F)}\times W^*}_{Y {(F)}})$.
Considering $\frac{1}{p}$ as a map $Y \to \mathbb{P}^1=\bar{\A}^1$, we also get
$\eta_{\frac{1}{p},\omega}\in\Sc^*(Y {(F)}\times F)$ and
$ {\hat \eta}_{\frac{1}{p},\omega}\in \RamiF{\Sc^*}(Y {(F)}\times F,
\DrinD{D^{Y(F)\times F}_{Y (F)}})$.
Then
\begin{equation}   \label{e:equality_in_question}
 {\hat \eta}_{\phi,\omega}  =g^*( {\hat \eta}_{\frac{1}{p},\omega}),
\end{equation}
where 
$g:Y \times W^{*} \to Y \times F$ is defined by
\DrinD{
\begin{equation}   \label{e:def_of_g}
g(y,\xi ):=(y,\langle \xi , {\alpha}(y)\rangle \,), \quad\quad y\in Y,\;\xi\in W^*.
\end{equation}
}


\end{lemma}
}
Note that \emph{the map $g$ is a submersion} (because $ {\alpha}$ has no zeros),  so we have a well-defined map
$g^*:  \RamiF{\Sc^*}(Y\RamiM{(F)}\times F,\DrinD{D^{Y(F)\times F}_{Y (F)}} )\to \RamiF{\Sc^*}(Y\RamiM{(F)}\times W^*,
\DrinD{D^{Y(F)\times W^*}_{Y (F)}})$ \DrinH{and the r.h.s. of \eqref{e:equality_in_question} makes sense.}



\begin{rem}
If $p$ has no zeros then Lemma~\ref{lem:red_to_dim_1} is obvious. To see this, note that by
Lemma~\ref{lem:exp_comp}, in this case
$$ {\hat \eta_{\phi,\omega}}=f_{\phi} \cdot \RamiS{pr_{W}}^{*}(|\omega|), \quad
g^*( {\hat \eta_{\frac{1}{p},\omega}})=\RamiS{g^*(f_{\frac{1}{p}} \cdot pr_{F}^{*}(|\omega|))=g^*(f_{\frac{1}{p}}) \cdot pr_{W}^{*}(|\omega|)},$$
where $f_{\phi}:Y\RamiM{(F)}\times W^*\to\C$ and $f_{\frac{1}{p}}:Y\RamiM{(F)}\times F\to\C$ are defined by
\begin{equation}   \label{e:fphi}
f_{\phi}(y,\xi )=\psi  {\left(  \frac{\langle\xi,  {\alpha}(y)\rangle}{p(y)}  \right)} ,\quad\quad y\in Y,\;\xi\in  W^*,
\end{equation}
\begin{equation}   \label{e:f1/p}
f_{\frac{1}{p}}(y,\nu )=\psi  {\left(  \frac{\nu}{p(y)}  \right)} ,\quad\quad y\in Y,\;\nu\in  F,
\end{equation}
so \eqref{e:equality_in_question} follows from the equality $f_{\phi}=g^*(f_{\frac{1}{p}})$,
which is obvious by \eqref{e:def_of_g}, \eqref{e:fphi}, and \eqref{e:f1/p}.
(The case where $p$ has zeros is not much harder in view of Remark \ref{rem:exp_comp}.)
\end{rem}

\RamiR{Let us give a complete proof now.}

\begin{proof}[\RamiR{Proof of Lemma \ref{lem:red_to_dim_1}}]
\RamiF{
Let $Y_0=p^{-1}(F-\DrinD{ \{ 0\} } ),\omega_0=\omega|_{Y_0}$,  $i_{\frac{1}{p}}:Y_0 \to Y \times F$ be the graph of $\frac{1}{p}$ and
$i_{\phi}:Y_0 \to Y \times W$ be the graph of $\DrinD{\phi }$.

\RamiR{Recall that $Mon(F,W)$ \DrinH{stands for} the space of monomorphisms from $F$ to $W$.} Let $\nu:Y \to Mon(F,W)$ be given by $\nu(y)(\lambda)=\lambda \cdot {\alpha}(\DrinD{y} )$.
Let $\rho_\nu:Y \times  F \to Y \times  W$ be the corresponding map (as in \DrinD{N}otation \ref{not:F_prop}).
\RamiR{The map $i_{\phi}$ \RamiS{is equal to} the composition
$$Y_{\RamiT{0}} \overset{ i_{\frac{1}{p}}}{\to} Y \times  F \overset{\rho_\nu}{\to} Y \times  W.$$
Thus,}
$$(\rho_\nu)_* \eta_{\frac{1}{p},\omega}= (\rho_\nu)_* ( (i_{\frac{1}{p}})_*(|\omega_0|))=
(i_{\phi})_*(|\omega_0|)= \eta_{\phi,\omega}\, .$$
\RamiR{Note that $\rho_{\nu^t}=g$.} Thus by Proposition \ref{prop:F_prop2},
$${\hat \eta_{\phi,\omega}} = (\rho_{\nu^t})^*  (\cF^{*}_{F}(\eta_{\frac{1}{p},\omega}))
 =g^*(\cF^{*}_{F}(\eta_{\frac{1}{p},\omega}))\DrinD{\, .}$$
 }
\end{proof}

\begin{lemma}  \label{lem:dim_1}
Let $Y$ be the affine space with coordinates $y_1, \dots, y_n.$ Let $p:Y \to F$ \RamiAS{be} defined by
$p= \prod_{i=1}^n y_i^{l_i}$,
where
$l_i \in \Z_{\geq 0}$. Let $\omega$ be the top differential form on $Y$
given by $\omega= (\prod_{i=1}^n y_i^{r_i}) d{y_1} \wedge \dots \wedge d{y_n}$, where $r_i \in \Z$.
 Suppose $r_i \geq 0$ whenever $l_i=0$,
\RamiG{
so $\omega$ is regular on the set $Y_0:=\{ y\in Y | p(y)\neq 0\}$ and therefore
$\eta_{\frac{1}{p},\omega}$ is well-defined.
}

Then $\RamiT{\hat \eta}_{\frac{1}{p},\omega}$ is  {\RamiJ{WF-holonomic}}.
\end{lemma}
\RamiG{
\RamiR{This lemma} follows from the next \DrinH{one} combined with Corollary \ref{cor:WF_Hol2}.
\begin{lemma}\label{lem:equv}
\DrinH{In the situation of Lemma~\ref{lem:dim_1} one has}
$$\pi(\DrinH{\alpha_1, \cdots,\alpha_n})(\RamiT{\hat \eta}_{\frac{1}{p},\omega})=|\prod_{i=1}^n \alpha_i^{-1-r_i}|\RamiT{\hat \eta}_{\frac{ 1}{p},\omega}, \quad \DrinH{(\alpha_1, \cdots,\alpha_n)\in (F^{\times})^n ,}$$
\DrinH{where $\pi$ denotes the following action of $(F^{\times})^n$ on $Y \times  F$:}
$$\pi(\alpha_1, \cdots,\alpha_n)\cdot (y_1, \cdots,y_n,\xi)\DrinH{:= }
(\alpha_1 y_1, \cdots,\alpha_n y_n, \xi \prod_{i=1}^n \alpha_i^{l_i}).$$
\end{lemma}

\DrinD{
\begin{remark}
 By Lemma \ref{lem:exp_comp} and Remark \ref{rem:exp_comp},
$$\RamiT{\hat \eta}_{\frac{1}{p},\omega}=g\cdot |d{y_1} \wedge \dots \wedge d{y_n}|,$$
where $g$ is the function
$$g(y,\xi )=\psi(\xi \cdot \prod_{i=1}^n y_i^{-l_i} ) \cdot
\prod_{i=1}^n |y_i|^{r_i} $$
considered 
 as a \emph{generalized} function on the \emph{whole} $Y\RamiM{(F)} \times F$ (namely, to compute its scalar product with
 any test function, one integrates first with respect to $\xi$ and then with respect to $y$). So
 Lemma~\ref{lem:equv} just says that the equality
$$g(\alpha_1^{-1}y_1,\ldots ,\alpha_n^{-1}y_n,\xi\prod_{i=1}^n \alpha_i^{-l_i})
=\prod_{i=1}^n |\alpha_i|^{-r_i}\cdot g(y_1,\ldots y_n,\xi )$$
holds in $\G (Y\RamiM{(F)}\times F)$ (not merely on the locus $y_i\ne 0$).
This is clear. On the other hand, a formal proof of the lemma is given below.
\end{remark}

}
}

%
\begin{proof}[\RamiG{Proof of Lemma \ref{lem:equv}}]
Consider the action $\pi_1$ of $T:=(F^\times)^n$ on $Y$ given by
$\pi_1(\alpha_1, \cdots,\alpha_n)\cdot (y_1, \cdots,y_n)=  (\alpha_1 y_1, \cdots,\alpha_n y_n).$
 Let $t=(\alpha_1, \cdots,\alpha_n)$.
Clearly,
$$\pi_1(t)(p)=( \prod_{i=1}^n \alpha_i^{-l_i})\cdot p\quad  \quad \text{ and } \quad \quad \pi_1(t)(\omega)=(\prod_{i=1}^n \alpha_i^{{-1-r_i}})  \cdot \omega.$$
Thus by
Proposition \ref{prop:F_prop2},  we have
\begin{multline*}
\pi_1(t)\cF^{*}_F(\eta_{\frac{1}{p},\omega}) =\cF^{*}_F(\pi_1(t)^{}\eta_{\frac{1}{p},\omega})=\cF^{*}_F(\eta_{\frac{1}{\pi_1(t)^{}(p)},\pi_1(t)^{}(\omega)})=
\\ =\cF^{*}_F\left(\eta_{\frac{ \prod_{i=1}^n \alpha_i^{l_i}}{p},(\prod_{i=1}^n \alpha_i^{-1-r_i})\omega}\right)=
   \RamiE{|\prod_{i=1}^n \alpha_i^{-1-r_i}|\cdot \cF^{*}_F(\rho_{\prod_{i=1}^n \alpha_i^{l_i}}(\eta_{\frac{1}{p},\omega}))=}
\\
\RamiE{= |\prod_{i=1}^n \alpha_i^{-1-r_i}|\cdot \rho_{\prod_{i=1}^n \alpha_i^{-l_i}}\left(\cF^{*}_F(\eta_{\frac{1}{p},\omega})\right).}
\end{multline*}
\RamiG{This implies
$$\pi(t)^{}\cF^{*}_W(\eta_{\frac{1}{p},\omega})=|\prod_{i=1}^n \alpha_i^{-1-r_i}|\cdot \cF^{*}_W(\eta_{\frac{ 1}{p},\omega}).$$
}

\end{proof}

\subsubsection{\DrinH{Proof} \RamiF{of  {Proposition} \ref{thm:main_c}}}
\DrinH{Let us} introduce the following \RamiG{\emph{ad hoc}} terminology.
\begin{definition}
A quadruple $(Y,W,\phi,\omega)$ as in Notation \ref{not:eta} is said to be
``good" if  $\RamiT{\hat \eta}_{\phi,\omega} \in \RamiF{\Sc^*}( {Y(F)}\times W^*,
\DrinD{D^{Y(F)\times W^*}_{Y (F)}})$
is  {\RamiJ{WF-holonomic}}.
\end{definition}
\DrinH{Our goal is to show that any quadruple $(Y,W,\phi,\omega)$ satisfying the conditions of Proposition~\ref{thm:main_c} is good. We will need  the following  obvious lemma.}
\begin{lemma} \label{lem:loc_hom}
Let $(Y,W,\phi,\omega)$ be as above.
\RamiR{Let  $e:U \to Y$ be an {\et} map and $f\in O^{\times}(Y)$ be \DrinH{an} invertible  regular function. Then
\begin{enumerate}
 \item \label{lem:loc_hom:1}
$\eta_{\phi,f\omega}= |f| \cdot \eta_{\phi,\omega}.$
 \item \label{lem:loc_hom:2} \DrinH{Let} $\rho_{f}:Y \times W \to Y \times W $ \DrinH{denote} the homothety action as in \DrinH{N}otation \ref{not:F_prop}. Then
$$\eta_{f \phi,\omega}= \rho_{f}(\eta_{ \phi,\omega}).$$
\item \label{lem:loc_hom:3}  \RamiS{Let $e^*(\phi)$ denote the composition $$U \overset{e}{\to} Y \overset{\phi}{\to}\overline W$$

Then}
$$\eta_{e^*(\phi),e^*(\omega)}= (e \times Id_W)^*( \eta_{ \phi,\omega})$$

\end{enumerate}
}
\end{lemma}

\DrinH{Let us now study how the property of being good  depends on $(Y,W,\phi,\omega)$.}

\begin{proposition}[Locality] \label{prop:loc}
Let $(Y,W,\phi,\omega)$ be as above.
 {
\begin{enumerate}
 \item \label{prop:loc:1} Let $Y=\bigcup U_i$ be a Zariski open cover of $Y$. Suppose that
 the quadruple $(U_i,W,\phi|_{U_i},\omega|_{U_i})$ is good for each $i$ .
Then the quadruple $(Y,W,\phi,\omega)$ is good.
\item \label{prop:loc:2}Let $e:U \to Y$ be an  {\et} map. Suppose that the quadruple $(Y,W,\phi,\omega)$ is good. Then the
 quadruple $(U,W,\RamiS{e^*(\phi)},e^*(\omega))$ is good.
\end{enumerate}
}
\end{proposition}
\begin{proof}
$ $
\begin{enumerate}
 \item By Lemma \ref{lem:loc_hom}\RamiS{\eqref{lem:loc_hom:3}}, we have $\eta_{\phi|_U,\omega|_U}= (  \eta_{ \phi,\omega})|_{U \times W}$.
By Lemma \ref{prop:F_prop}, we get  {$\hat \eta_{\phi|_U,\omega|_U}=  (\hat \eta_{ \phi,\omega})|_{U(F) \times W}$}.
 By Proposition \ref{prop:WF_prop}, this gives
 {$\WF(\hat \eta_{\phi|_U,\omega|_U})= \WF(\hat \eta_{ \phi,\omega})\cap (T^*(U \times W))(F)$}.
This immediately implies the assertion.
\item  {Follows immediately from Lemma \ref{lem:loc_hom}\RamiS{\eqref{lem:loc_hom:3}} and Proposition \ref{prop:WF_Hol}. }
\end{enumerate}
\end{proof}

\begin{proposition}[\RamiG{Homogeneity}] \label{prop:homo}
Let $(Y,W,\phi,\omega)$ be a good quadruple. Let $f_1,f_2 \in O^{\times}(Y)$ be invertible \RamiJ{regular} functions. Then $(Y,W,f_1\phi,f_2 \omega)$ is good.
\end{proposition}
\begin{proof}
By  Lemma \ref{lem:loc_hom}\RamiS{(\ref{lem:loc_hom:1}},\RamiS{\ref{lem:loc_hom:2})}, we have $\eta_{f_1 \phi,f_2 \omega}= (\rho_{f_1}(|f_2| \cdot \eta_{ \phi,\omega})$. By \RamiE{Proposition}  \ref{prop:F_prop2}, we get
\RamiE{$\cF_W^*(\eta_{f_1 \phi,f_2 \omega})= |f_2| \cdot\rho_{f_1^{-1}} (\cF_W^*(\eta_{ \phi,\omega}))$}.
 By Proposition \ref{prop:WF_prop}, this gives $\WF(\cF_W^*(\eta_{f_1 \phi,f_2 \omega}))\subset \rho_{f_1^{ }}^*(\WF(\cF_W^*((  \eta_{ \phi,\omega}))))$ . This immediately implies the assertion.
\end{proof}
 {
\begin{corollary}  \label{cor:homo}
Let $(Y,W,\phi,\omega)$ be a quadruple as above. \DrinH{Assume that $\dim W=1$, so we can interpret $\phi$ as a rational function on $Y$. Then} the property of being good depends only on the divisors of
$\phi$ and $\omega$.
\end{corollary}


We will also need the following standard lemma.
\begin{lemma} \label{lem:mon}
Let $D \subset U$ be an SNC divisor inside a smooth algebraic variety. Let $D_i$ be a collection of divisors in $U$ supported\footnote{\RamiR{Recall that a divisor on a manifold is an integral linear combination of
irreducible subvarieties of codimension $1$. The words ``supported in $D$"
mean that each of the subvarieties is contained in $D$.}} in $D$.
  Then there exist
 a Zariski cover $V_j$ of $U$,  {\et} maps $e_j:V_j \to \A^n$  and divisors $D_{ji}$ in $\A^n$ such that
 $(D_i)|_{V_j}=e_j^*(D_{ji})$ and  \DrinH{each} $D_{ji}$ \DrinH{is} supported \RamiM{in} the union of the coordinate hyperplanes.
\end{lemma}
}
\RamiR{
Now we are ready to prove the following particular case of
Proposition \ref{thm:main_c}.
\begin{lemma}\label{lem:sp_case}
Let $(Y,W,\phi,\omega)$  be as in  Proposition~\ref{thm:main_c}. Moreover,
suppose that $W=F$ and
$0\not\in\phi (Y)$. Then $(Y,W,\phi,\omega)$  is good.
\end{lemma}}
\begin{proof}
\RamiR{
The proof is based on Lemma \ref{lem:dim_1}.
We can rewrite $\phi=1/p$, \DrinH{where} $p$ is a regular function on $Y$.} We know that 
 \RamiR{$\Xi:= Y_\infty  \cup Z$} is \RamiM{an} {SNC} divisor in
  \RamiR{$Y$.}

Let $D_1$ be the divisor of $p$ and $D_2$ the divisor of 
\RamiR{$\omega$}.
By Lemma \ref{lem:mon}, we can find a Zariski cover  $V_j$ of
\RamiR{$Y$}, {\et} maps $e_j:V_j \to \A^n$,
and  divisors $D_{\RamiI{j1}},D_{\RamiI{j2}}$ such that
$(D_i)|_{V_j}=e_j^*(D_{ji})$ and the divisors $D_{ji}$ are supported \RamiM{in} the union of the coordinate hyperplanes.
By Proposition \ref{prop:loc}\RamiS{\eqref{prop:loc:1}}, it is enough to show that for any $j$
the quadruple  $(V_j,F,\frac{1}{p|_{V_j}},(\omega|_{V_j}))$ is good.
By Corollary \ref{cor:homo}, we may replace $p|_{V_j}$ and $\omega|_{V_j}$ by any other function and form with the
same divisor.
 So by Proposition \ref{prop:loc}\RamiS{\eqref{prop:loc:2}}, it is enough to show that
$(\A^n,F,\frac{1}{q},\eps)$ is good 
\RamiR{for some regular function $q$  and top differential form $\eps$ on $\A^n$ such that} the divisor of $q$ is
 $D_{\RamiI{j1}}$ and the divisor
of $\eps$ is $D_{\RamiI{j2}}$.
This follows from Lemma  \ref{lem:dim_1}.
\end{proof}

\begin{proof} [Proof of  {Proposition} \ref{thm:main_c}]
$ $
Without loss of generality we may assume that $Y$ is irreducible and $\omega \neq 0$.
%
%

We have to show that $(Y,W,\phi,\omega)$  is good.
 {We can cover $Y$ by open subsets $U_i$ so that $\phi|_{U_i}=(f_i(x):p_i(x))$,
 where for each $i$, one of the maps $f_i:U_i \to W$ and $p_i:U_i \to F$ never vanishes.}

By Proposition \ref{prop:loc}\RamiS{\eqref{prop:loc:1}},
 it is enough to show that $(U_i,W,\phi|_{U_i},\omega|_{U_i})$ is good.
\begin{itemize}
 \item The case \RamiI{when}  {$p_i$ never vanishes}.\\
By Lemma \ref{lem:exp_comp} \RamiAQ{(and the fact that a pullback of a WF-holonomic distribution is WF-holonomic -- see Proposition \ref{prop:WF_Hol}\eqref{prop:WF_Hol:1})} it is enough to show that $|\omega|$ is \RamiJ{a WF-holonomic} distribution on $U_i$.
For this it is enough to show that
$(U_{ {i}},F,1,\omega|_{U_ {i}})$ is good.
\RamiR{This follows from }
\DrinH{Lemma \ref{lem:sp_case}.}
 \item \RamiF{The case  \RamiI{when}  {$f_i$ never vanishes}.}\\
 By Lemma \ref{lem:red_to_dim_1}, there exists a submersion
$g:U_{i} \times W^\RamiAQ{*} \to U_{i} \times F$ such that
 $\RamiT{\hat \eta}_{\phi,\omega}  =g^*({\hat \eta_{\frac{1}{{p_i}},\omega}}).$
So by Proposition \ref{prop:WF_Hol}, it is enough to show that $(U_i,F,\frac{1}{{p_i}},\omega|_{U_i})$ is good.
\RamiR{This again follows from }
\DrinH{Lemma \ref{lem:sp_case}.}
\end{itemize}
\end{proof}


%




\section{Proof of Theorem  \ref{thm:futerO}}\label{sec:ex_boun}
\setcounter{lemma}{0}


\subsection{A fact from symplectic geometry}
We will need the following lemma, which will be proved in Appendix \ref{app:symp_geom}.


\RamiCPD{
 \begin{lemma}    \label{l:symplectic_lem}
Let $\RamiAJ{W}$ be a finite dimensional vector space over $F$ and $X$ be a manifold. Let $E$ be a vector bundle over $\RamiAJ{Y}$ which is a subbundle of the trivial vector bundle $\RamiAJ{Y} \times \RamiAJ{W}$. Let $E^\bot\subset \RamiAJ{Y} \times \RamiAJ{W}^*$ be its orthogonal complement.
Then $CN_{E^\bot}^{\RamiAJ{Y} \times \RamiAJ{W}^*}=CN_E^{\RamiAJ{Y} \times \RamiAJ{W}}$.

Here
$$CN_E^{\RamiAJ{Y} \times \RamiAJ{W}} \subset T^*(\RamiAJ{Y} \times \RamiAJ{W})=T^*(\RamiAJ{Y}) \times \RamiAJ{W} \times \RamiAJ{W}^*,$$
$$CN_{E^\bot}^{\RamiAJ{Y} \times \RamiAJ{W}^*} \subset T^*(\RamiAJ{Y} \times \RamiAJ{W}^*)=T^*(\RamiAJ{Y}) \times \RamiAJ{W}^* \times \RamiAJ{W},$$
and the symplectic manifolds $T^*(\RamiAJ{Y}) \times \RamiAJ{W} \times \RamiAJ{W}^*$ and $T^*(\RamiAJ{Y}
) \times \RamiAJ{W} \times \RamiAJ{W}^*$ are identified via the map
$$\RamiAJ{W} \times \RamiAJ{W}^*{\buildrel{\sim}\over{\longrightarrow}} \RamiAJ{W}^* \times \RamiAJ{W}, \quad \RamiAJ{(w,\phi)\mapsto (\phi,-w)}.$$
 \end{lemma}
}

\subsection{Some notation}  \label{ss:notation_explicit}
$ $

\RamiAK{
\begin{notn} $ $
\begin{enumerate}
\item
Let $W$ be a vector space. We denote by $\mathbb{P}(W)$  the {projective space whose points are
1-dimensional subspaces} in $W^*$. We denote by $Taut_{\mathbb{P}(W)}$ the tautological  {line} bundle of $\mathbb{P}(W)$ which is {a} subbundle of the {trivial} bundle with fiber $W^*$.

\item
\RamiAM{Recall that
\RamiP{$\overline{W}:=\mathbb{P}(W^* \oplus F)$}.
}
\item
Let $W$ be a finite dimensional vector space over $F$ and $X$ be a manifold. Let $E$ be a {vector bundle} over $X$ which is a subbundle of the {trivial} bundle $X \times W$. Then
its orthogonal complement in $X \times W^*$ is denoted by $E^\bot$.

In particular, we denote by $Taut_{\mathbb{P}(W)}^\bot$ the orthogonal complement to $Taut_{\mathbb{P}(W)}\,$, which is a co-dimension $1$ subbundle of the {trivial} bundle $\mathbb{P}(W) \times W$.
\end{enumerate}
\end{notn}
}
\begin{notation}
  \label{notation:S}
Given a morphism $f:M\to N$ between algebraic manifolds, define $\RamiAJ{\Crit}_f\subset T^*N$ by
$\RamiAJ{\Crit}_f:=f_*(M)$ where $M\subset T^*M$ is the zero section and $f_*$ is as in Definition~\ref{def:dir_and_inv_image}(\ref{def:dir_and_inv_image:3}).
\end{notation}

\begin{rem}   \label{r:S isotropic}
By Lemma~\ref{lem:iso_dir_invs_im}, $\RamiAJ{\Crit}_f$ is isotropic. It is easy to see that if $f:M\to N$ is proper then
the subset $\RamiAJ{\Crit}_f\subset T^*N$ is closed.
\end{rem}

\begin{rem}
The fiber of $\RamiAJ{\Crit}_f$ over $y\in N$ is nonzero if and only if $y$ is a critical value for $f:M\to N$.
\end{rem}

\begin{rem}
Even if $F=\mathbb{C}$, it can happen that $\Crit_f$ is not Lagrangian (a simple example is given in\cite[Appendix A]{Dri}).
\end{rem}

\RamiCPD{
\begin{notation}   \label{n:SNC_notation}
For an SNC divisor $D$ on some algebraic manifold,
let $\hat{D_1}$ denote the disjoint union of the irreducible components of $D$, let $\hat{D_2}$ denote the disjoint union of the pairwise intersections of the irreducible components of $D$, and so on. Let $\hat{D}$ denote the disjoint union of $\hat{D_i}$, $i\ge 1$. Clearly
$\hat D$ is smooth, and if $D$ is projective then so is $\hat{D}$.

\end{notation}
}

%
%

\subsection{An explicit upper bound for the wave front}
Let $F$ be a local field of characteristic 0. Because of the numerous \RamiAJ{references} to \S\ref{sec:pf_main_bb}, the reader may assume for a while that $F$ is non-Archimedean. However, it will be clear from  \RamiAJ{\S\S}\ref{ssec:arc_pf} that this assumption is not necessary either in  \S\ref{sec:pf_main_bb} or here.

\subsubsection{The goal} \label{sss:thegoal}
Consider the following setting (it is essentially\footnote{
{Only the notation is slightly different: here we denote by $X$ and $\phi$ the objects that were denoted by $\tilde X$ and $\bar \phi \circ \rho$ in the proof of Theorem \ref{t:C}}.} the same as in the proof of Theorem \ref{t:C} given at the end of 
\RamiAJ{\S\S}\ref{ss:Reduction to key}).

Let $X,Y$ be  {algebraic} manifolds over $F$ and $W$ a vector space over $F$,  {with}  $\dim W<\infty$.
Let $\phi :X\to Y \times \overline{W}$ be a proper map.  \RamiT{Let $X_0:=\phi^{-1}(Y \times W)$.} Let $\omega$ be a top differential form on $X\RamiT{_0}$.
 Let $Z$ be the \RamiT{closure of the} zero set of $\omega$ \RamiT{in $X$}.
 Let $D:=Z \cup \phi^{-1}((\overline{W}-W)\times Y)$.
Assume that $D$ is an SNC divisor.

Then we have the distribution $(\phi|_{X_0})_*(|\omega|))$ on $Y(F)\times W$.
In \S\ref{sec:pf_main_bb} we proved that its partial Fourier transform
$\cF^*(\RamiT{(\phi|_{X_0})}_*(|\omega|)))$ is holonomic, which means that
\begin{equation}   \label{e:property_needed}
WF(\cF^*(\RamiT{(\phi|_{X_0})}_*(|\omega|)))\subset L(F)
\end{equation}
for \emph{some} isotropic algebraic subvariety $L\subset T^*(Y \times  W^{*})$.

Our goal now is to describe a \emph{specific} $L$ with property \eqref{e:property_needed}. The
definition of $L$ given below is purely algebro-geometric, so the fact that this $L$ satisfies
\eqref{e:property_needed} will imply Theorem~\ref{thm:futerO}.

\subsubsection{Definition of $L$}   \label{sss:Definition of L}
Let \RamiAJ{$\hat D$} be as in Notation~\ref{n:SNC_notation}.
Let \RamiAJ{$\hat D'$} be the union of those components of \RamiAJ{$\hat D$} whose image in \RamiAJ{$ D$} is contained in $\phi^{-1}(\overline{W}-W)$\RamiT{.}

Let $\pi:X\to \overline{W}$ and $\tau:X\to Y$ be the  compositions of $\phi :X\to Y \times \overline{W}$ with the projections $Y \times \overline{W}\to \overline{W}$ and $Y \times \overline{W}\to Y$.
{
Let
$$\pii:\RamiAJ{\hat D'} \to\overline{W}-W= \mathbb{P}(W^*)$$
denote the map induced by $\pi:X\to \overline{W}$.
Set
}
 $\RamiAJ{E}:=\pii^{*}(Taut_{\mathbb{P}(W^{*})})$.
 Recall (see \RamiAK{\S\S\ref{ss:notation_explicit}}) that
$Taut_{\mathbb{P}(W^{*})}\subset \mathbb{P}(W^{*})\times W$; accordingly,
 we \RamiAL{have a map $E \to X \times W$} .
Set 
\begin{equation}   \label{e:tDdef}
\RamiAL{\tD:=}(X\times 0) \sqcup\RamiAJ{ (\hat{D} \times 0)} \sqcup{E}\, ;
\end{equation}
where $\sqcup$ stands for the disjoint union; clearly
$\RamiAL{\tD}$ is a manifold equipped with a natural map
$\RamiAL{\mmup}:\RamiAL{\tD} \to X \times W$.
Let $\RamiAL{\mmu}:\RamiAL{\tD} \to Y \times W$ be the composition $$\RamiAL{\tD}  \overset{\RamiT{\RamiAL{\mmup}}}{\to} X \times W\overset{\tau \times Id}{\to}Y \times  W.$$
Let $\RamiAJ{\Crit}_{\RamiAL{\mmu}}$ be as in Notation  \ref{notation:S}. By Remark~\ref{r:S isotropic},
$\RamiAJ{\Crit}_{\RamiAL{\mmu}}\subset T^*(Y \times  W)$ is an isotropic closed algebraic subvariety.

Finally, define $L\subset T^*(Y \times  W^{*})$ to be the image of $\RamiAJ{\Crit}_{\RamiAL{\mmu}}$ under the symplectic isomorphism $T^*(Y \times  W){\buildrel{\sim}\over{\longrightarrow}} T^*(Y \times  W^*)$ from Lemma \ref{l:symplectic_lem}.
\begin{thm} \label{exp}
We keep the notation of \S\ref{sss:thegoal}. Let $L$ be as in \S\ref{sss:Definition of L}.  Then
$WF(\cF^*(\RamiT{(\phi|_{X_0})}_*(|\omega|)))\subset L(F)$.
\end{thm}

Before proving \RamiW{Theorem  \ref{exp}}, {let us formulate two corollaries in the case that $Y$ is a point. In this case the wave front in question is a subset of $T^*(W^*)=W^*\times W$, and the next corollary gives an upper bound for its intersection with $\RamiAO{(W^*-0)}\times (W-0)$.
}

\begin{cor}\label{cor:exp}
Let $X$ be  an {algebraic} manifold and $W$ a vector space over $F$,  with  $\dim W<\infty$.
Let $\phi:X\to \overline{W}$ be a proper map.  \RamiT{Let  $X_0:=\phi^{-1}(\RamiAJ{ W})$.} Let $\omega$ be a top differential form on $X\RamiT{_0}$. Let $Z$ be the \RamiT{closure of the} zero set of $\omega$  \RamiT{in $X$}. Let $D=Z \cup \phi^{-1}(\overline{W}-W)$.
Assume that $D$ is an SNC divisor.

Let $\RamiAJ{\hat D}$ be as in Notation~\ref{n:SNC_notation}.
Let $\RamiAJ{\hat D'}$ be the union of those components of $\RamiAJ{\hat D}$ whose image in $\RamiAJ{D}$ is contained in \RamiAJ{$X-X_{0}$}. 
Let 
$\RamiAL{\pii}:\RamiAJ{\hat D'}\to \overline{W}-W=\mathbb{P}(W^*)$ be the natural map.
{Define $\mathrm{PCrit}_{{\pii}}\subset\mathbb{P}(W^*)\times\mathbb{P}(W)$ to be the set of pairs $(z,H)$, where $z\in \mathbb{P}(W^*)$ and
$H\subset \mathbb{P}(W^*)$ is a projective hyperplane\footnote{Recall that a point of $\mathbb{P}(W)$ is the same as a hyperplane $H\subset\mathbb{P}(W^*)$.} containing $z$ such that $\pii :\hat D'\to \mathbb{P}(W^*)$ is not transversal\footnote{By definition, non-transversality of $\pii :\hat D'\to \mathbb{P}(W^*)$ to $H$ at $x\in \pii^{-1}(H)$
means that the image of $d_x\pii :T_x\hat D'\to T_{\pii (x)}\mathbb{P}(W^*)$ is contained in $T_{\pii (x)}H$.
}
to $H$ at some point of  $\pii^{-1}(z)$. Let $L'\subset (W^*-0)\times (W-0)$ denote the preimage of $\mathrm{PCrit}_{{\pii}}$ with respect to the map
$$(W^*-0)\times (W-0)\to\mathbb{P}(W)\times\mathbb{P}(W^*)=\mathbb{P}(W^*)\times\mathbb{P}(W).$$
Then
\begin{enumerate}
\item[(a)] \RamiAJ{$L' \subset (W^*-0) \times(W-0)$ is  an  isotropic closed algebraic subvariety.}
\item[(b)]
$WF(\RamiT{\cF^*((\phi|_{X_0})_*(|\omega|))})\cap ((W^*-0) \times(W-0)) \subset \RamiT{L'(F)}.$

\end{enumerate}
}
\end{cor}

{
\begin{rem}   \label{r:PCrit}
The set $\mathrm{PCrit}_{{\pii}}\subset\mathbb{P}(W^*)\times\mathbb{P}(W)$ introduced above is the ``projectivization" of the set $\Crit_{\pii}\subset T^*(\mathbb{P}(W^*))$ from
Notation~\ref{notation:S}. More precisely, $\mathrm{PCrit}_{{\pii}}$ canonically identifies with the quotient
of $\Crit_{\pii}-\mbox{\{zero section\}}$ by the action of $\mathbb{G}_m\,$.
\end{rem}
}

{
\begin{proof}
Let $p:W-0\to \mathbb{P}(W^{*})$ be the canonical map. We have an isotropic closed algebraic subvariety
\begin{equation}    \label{e:p*Crit}
p^*(\Crit_{\pii})\subset T^*(W-0)=(W-0)\times W^*=W^*\times (W-0).
\end{equation}
By Remark~\ref{r:PCrit} and the definition of $L'$, we have
\begin{equation}   \label{e:formula for L'}
L'=p^*(\Crit_{\pii})\cap ((W^*-0)\times (W-0)),
\end{equation}
which proves statement (a).

Let $\RamiAJ{\tD}$ and $\RamiAL{\mmu}{:\tD\to Y\times W} = W$ be as in \S\ref{sss:Definition of L}.
By  Theorem \ref{exp},
\begin{equation}   \label{eq:2WF_bnd}
WF(\RamiT{\cF^*((\phi|_{X_0})_*(|\omega|))})\cap (W^* \times (W-0))
\subset \Crit_{\mmu} (F)\cap ((W-0) \times W^*)
\end{equation}
(here we  identify $W \times W^*$ with $W^* \times W$, just as in formula \eqref{e:p*Crit}).

Our $\tD$ was defined by formula~\eqref{e:tDdef} to be a disjoint union of three sets. It is clear that
$\mmu^{-1}(W-0)$ is contained in the third one, denoted by $E$. Moreover, one has a Cartesian square
\[
\xymatrix
{       \mu^{-1}(W-0) \ar@{->}^{}[r] \ar@{->}^{}[d]& W-0
 \ar@{->}^{p}[d]  \\
        \hat D' \ar@{->}^{\pii}[r]                                 & \mathbb{P}(W^*)  \\
}
\]
So the r.h.s. of  \eqref{eq:2WF_bnd} equals $p^*(\Crit_{\pii})$. Thus we see that
\[
WF(\RamiT{\cF^*((\phi|_{X_0})_*(|\omega|))})\cap ((W^*-0) \times (W-0))\subset p^*(\Crit_{\pii})
\cap ((W^*-0) \times (W-0)).
\]
Combining this with \eqref{e:formula for L'}, we get statement (b).
\end{proof}
}

\begin{cor}   \label{c:explcit_open}
Let $X,W,\phi,\omega,{\hat D'}$ and $\RamiAL{\pii}$ be as in  Corollary \ref{cor:exp}.
Define $U$ to be the set of all $\ell \in W^*-0$ such that the map $\RamiAL{\pii}:{\hat D'}\to  \mathbb{P}(W^*)$ is transversal to the hyperplane $H_{\ell}\subset  \mathbb{P}(W^*)$ corresponding to $\ell$.
Then $\cF^*((\phi|_{X_0})_*(|\omega|))|_{U(F)}$ is smooth.




\end{cor}

It is clear that $U$ is Zariski open in $W^*- 0$. 


\begin{proof}
{
By Corollary~\ref{cor:exp}, $\cF^*((\phi|_{X_0})_*(|\omega|))$  is smooth on the open subset
\begin{equation} \label{e:open_of_smoothness}
(W^*-0)-q^{-1}(p(\mathrm{PCrit}_{{\pii}}))\subset W^*-0,
\end{equation}
where $q:W^*-0\to  \mathbb {P}(W)$ and $p:\mathbb {P}(W^*) \times \mathbb {P}(W) \to \mathbb {P}(W)$ are the projections and $\mathrm{PCrit}_{{\pii}}\subset\mathbb{P}(W^*)\times\mathbb{P}(W)$ is as in Corollary~\ref{cor:exp}.
The subset \eqref{e:open_of_smoothness} clearly equals $U$.
}
\end{proof}

\subsection{Proof of Theorem~\ref{exp}}
We will proceed in stages analogous to the stages of the proof of  theorem \ref{thm:main_bb}.

\begin{lemma}\label{lem:exp_bn}
Let $Y$ be an algebraic manifold.
Let $p :Y\to F$ be a regular function such that
 $p \neq 0$ on a dense subset \RamiT{$Y_0$}. Let $\omega$ be a \RamiT{rational} top differential form \RamiT{on $Y$ which is regular on $Y_0$}. Let $D$ be the union of the zero sets of $\omega$ and of $p$. Assume that $D$ is an SNC divisor. {Just as in Notation~\ref{not:eta}, let
$\eta_{\frac{1}{p},\omega}\in\Sc^*(Y(F)\times F)$ denote the pushforward of $|\omega |$ with respect to
the map $Y_0\to Y\times F$ given by $y\mapsto (y,p(y)^{-1})$.
Let
${\hat \eta}_{\frac{1}{p},\omega}\in {\Sc^*}(Y(F)\times F, {D^{Y(F)\times F}_{Y (F)}})$ be its partial Fourier transform.
}

Let $\RamiAJ{\hat D}$ be as in Notation~\ref{n:SNC_notation}.
\RamiT{Let $\RamiAJ{\hat D'}$ be the union of those components of $\RamiAJ{\hat D}$ whose image in $\RamiAJ{ D}$ is contained in} {the zero set of $p$.}

Let $\RamiAL{\widetilde H}:= (\RamiAJ{Y}\times F)  \sqcup  \RamiT{(\RamiAJ{\hat D}\times F)} \sqcup
 \RamiT{(\RamiAJ{\hat D}' \times 0)}$. We have a natural map $\RamiT{\nu}:\RamiT{\RamiAL{\widetilde H}} \to Y \times F$.

Then
$$WF({\hat \eta}_{\frac{1}{p},\omega})\subset \RamiAJ{\Crit}_{\RamiT{\nu}}(F)$$


\end{lemma}
\RamiT{
\begin{proof}
Just as in the proof of lemma \ref{lem:sp_case}, it suffices to consider the case where
%
%
$Y$ is an affine space and $p$ and $\omega$ are given by monomials. In this case the statement follows from
Lemma \ref{lem:equv}
and Proposition \ref{prop:WF_prop}\eqref{it:G_act}. \end{proof}
}


\begin{prop}\label{exp_1}
Let $Y$ be an  {algebraic} manifold and $W$ a vector space over $F$,  {with}  $\dim W<\infty$.
Let $\phi :Y\to \overline{W}$ be a regular map. \RamiT{Let $Y_0:=\phi^{-1}( W)$.}  Let $\omega$ \RamiT be a {rational} top differential form \RamiT{on $Y$ which is regular on $Y_0$}. Let $Z$ be the zero set of $\omega$. Let $D=Z \cup \phi^{-1}(\overline{W}-W)$. Assume that $D$ is an SNC divisor.
{Just as in Notation~\ref{not:eta}, let
$\eta_{\phi,\omega}\in\Sc^*(Y {(F)}\times W)$ denote the pushforward of $|\omega |$ with respect to
the map $Y_0\to Y\times W$ given by $y\mapsto (y,\phi (y))$.
Let
${\hat \eta}_{\phi,\omega}\in {\Sc^*}(Y {(F)}\times W^{*},{D^{Y(F)\times W^{*}}_{Y (F)}})$ be its partial Fourier transform.
}

Let $\RamiAJ{\hat D}$ be as in Notation~\ref{n:SNC_notation}.
Let $\RamiAJ{\hat D'}$ be the union of those components of ${\RamiAJ{\hat D}}$ whose image in $\RamiAJ{ D}$ is contained in
$\phi^{-1}(\overline{W}-W)$\RamiT{.}

Consider $\phi_{\RamiAJ{\hat D'}}$ as a map \RamiAJ{$\phi_{\hat D'}:\hat D' \to \mathbb{P}(W^*)$} and  \RamiAJ{$G:=\phi^{*}_{\hat D'}(Taut_{\mathbb{P}(W^{*})}^\bot)$} as a subvariety of $Y \times W^*$. Let
$\RamiAL{\widetilde H}:= (Y\times W^{*})\sqcup  \RamiT{(\RamiAJ{\hat D} \times W^{*})} \sqcup G
$.
We have a natural map $\RamiT{\nu}:\RamiAL{\widetilde H} \to Y \times W^{*}$.


Then
$WF({\hat \eta}_{\phi,\omega})\subset \RamiAJ{\Crit}_{\RamiT{\nu}}(F)$.
\end{prop}

\RamiP{
\begin{proof}
{We follow the proof of Proposition  \ref{thm:main_c}. The claim is local, so we can reduce the problem to the case when $\phi=(f:p)$ where $f:Y \to W$ and $p:Y\to F$ are regular functions and one of them never vanishes. Let us analyze the two cases:
\begin{itemize}
 \item The case when  $p$ never vanishes.\\
By Lemma \ref{lem:exp_comp}.
$$WF({\hat \eta}_{\phi,\omega}) =WF(f_{\phi} \cdot  {pr_{W}}^{*}(|\omega|))\RamiAQ{\subset} WF({pr_{W}}^{*}(|\omega|))\RamiAQ{\subset} {pr_{\RamiS{W}}}^{*}(WF(|\omega|)).$$
 The assertion follows now from Lemma \ref{lem:exp_bn} (after noticing that in this case $D_i'$ and $G_i$ are empty).

 \item \RamiF{The case  \RamiI{when}  {$f$ never vanishes}.}\\
 By Lemma \ref{lem:red_to_dim_1}, we have a submersion
$g:Y \times W^{{*}} \to Y \times F$ such that
 $\hat \eta_{\phi,\omega}  =g^*({\hat \eta_{\frac{1}{{p}},\omega}}).$
So by Proposition \ref{prop:WF_prop2} \eqref{prop:WF_prop2:2}, $WF(\hat \eta_{\phi,\omega})\subset g^*(WF({\hat \eta_{\frac{1}{{p}},\omega}}))$.
Let $\RamiAL{\widetilde H}^1$ and $\nu^1:\RamiAL{\widetilde H}^1 \to Y \times F$ be the variety $\RamiAL{\widetilde H}$ and map $\nu$ from Lemma \ref{lem:exp_bn}.  By Lemma \ref{lem:exp_bn}, it is enough to show that $\RamiAJ{\Crit}_\nu= g^*(\RamiAJ{\Crit}_{\nu^1})$.
Let \RamiAJ{$\hat{D}_{j} ,\hat D'_{j}\subset Y$ and $  G_{j}\subset Y \times W^*$} be the components of \RamiAJ{$\hat{D},\hat D'$ and $G$}.   Note that
$$\RamiAJ{\Crit}_{\nu^1}
=(Y\times F)\cup \bigcup CN_{\RamiT{\hat{D}_{\RamiAJ{j}}\ \times F} }^{Y\times F} \cup \bigcup CN_{\RamiAJ{\hat D'_{j}} \times 0}^{Y\times F}$$
and
\begin{equation}   \label{e:S_nu}
\RamiAJ{\Crit}_{\nu}
=(Y\times W^{*})\cup \bigcup CN_{\hat{D}_{\RamiAJ{j}} \times W^{*}}^{Y\times W^{*}} \cup \bigcup CN_{G_{\RamiAJ{j}}}^{Y\times W^{*}}
\end{equation}

The assertion follows now from the fact that $G_{\RamiAJ{j}}=g^{-1}(\RamiAJ{\hat D_{j}} \times 0) $.

\end{itemize}
}
\end{proof}

}

%

\begin{cor}   \label{c:based on symplectic lemma}
In the notations of Proposition \ref{exp_1},
let $\RamiAJ{E}:=\phi^{*}_{\RamiAJ{\hat D'}}(Taut_{\mathbb{P}(W^{*})})$ as
 a subvariety of $Y \times W$.
Let $\RamiAL{\tD}:= (Y\times 0)\sqcup  \RamiAJ{(\hat D \times 0)  \RamiAS{\sqcup} E}\,$.
We have a natural map $\RamiAL{\mmu}:\RamiAL{\tD} \to Y \times W$. Let
$\rho:T^*(Y\times W) {\buildrel{\sim}\over{\longrightarrow}} T^*(Y\times
 W^*)$
be the standard identification (\RamiAE{as in Lemma \ref{l:symplectic_lem}}).


Then
$WF({\hat \eta}_{\phi,\omega})\subset \rho(\RamiAJ{\Crit}_{\RamiAL{\mmu}})(F)$.

\end{cor}

\begin{proof}
{By Proposition~\ref{exp_1}, it suffices to}
\RamiAI{
show that $\RamiAJ{\Crit}_{\nu}
= \rho(\RamiAJ{\Crit}_{\RamiAL{\mmu}})$. Combining formula~\eqref{e:S_nu} with
Lemma \ref{l:symplectic_lem}, we get
$$\RamiAJ{\Crit}_{\nu}
=\rho\left( (Y\times 0)\cup \RamiAJ{\bigcup CN_{\hat{D}_{j} \times0}^{Y\times W} \cup \bigcup CN_{G_{j}^{\bot}}^{Y\times W}}\ \right).$$
Clearly, $\RamiAJ{E=G^{\bot}\subset \hat D \times W.}$ Thus
$ G_{\RamiAJ{j}}^{\bot}$ are the components of $\RamiAJ{E}$. Therefore,
$(Y\times 0)\cup \ \bigcup CN_{\hat{D}_{\RamiAJ{j}} \times0}^{Y\times W} \cup \bigcup CN_{G_{\RamiAJ{j}}^{\bot}}^{Y\times W} = \RamiAJ{\Crit}_{\RamiAL{\mmu}}\,$.
}
\end{proof}

From Corollary~\ref{c:based on symplectic lemma} one deduces the following statement (\RamiT{using Proposition \ref{prop:WF_prop2}\eqref{prop:WF_prop2:3} in the same way as in the proof of Theorem \ref{t:C}}):
\RamiT{
\begin{cor}
In the situation of Theorem \ref{exp}, we have
$$WF(\cF^*((\phi|_{X_0})_*(|\omega|)))\subset (\tau \times \id_{W^*})_{*}(\rho(\RamiAJ{\Crit}_\RamiAL{\mmup}))(F)$$
where $\RamiAL{\mmup}:\RamiAL{\tD} \to X \times W$ and $\tau:X\to Y$ are as in \S\ref{sss:Definition of L}\ and $\rho:T^*(Y\times W) {\buildrel{\sim}\over{\longrightarrow}} T^*(Y\times W^*)$ is the standard identification (\RamiAE{as in Lemma \ref{l:symplectic_lem}}). \qed
\end{cor}
}

Theorem \ref{exp} follows from this corollary in view of the following lemma:

\begin{lemma}
Let $X,Y$ be manifolds and $V$ be a vector space. Let $A \subset T^*(X \times V) = T^*(X \times V^*)$  be a subset. Let $\phi:X \to Y$ be a map. Then $(\phi \times Id_V)_*(A)=(\phi \times Id_{V^*})_*(A)$.

Note that in the left hand side $A$ is considered as a subset in $T^*(X \times V)$ and in the right hand side $A$ is considered as a subset in $T^*(X \times V^*)$. The equality is under the standard identification $T^*(Y \times V) = T^*(Y \times V^*)$.
\end{lemma}
\begin{proof}

The \RamiAE{lemma  follows from Remark \ref{rem:simp_gem} and} the equality
$$\Lambda_{\phi \times Id_V}=\Lambda_{\phi \times Id_{V^*}}\, ,$$
where $\Lambda_{\phi \times Id_V}$ and $\Lambda_{\phi \times Id_{V^*}}$ have the same meaning as in
Remark \ref{rem:simp_gem}.
\end{proof}


\section{Proof of Theorem \ref{thm:futer} \RamiAF{in the non-Archimedean case}}\label{sec:pf_futer}
{In this section we deduce Theorem \ref{thm:futer} from Theorem \ref{thm:futerO} assuming that the local field $F$ is non-\RamiAS{Archimedean}. A slight modification of the same argument allows to prove Theorem~\ref{thm:futer} in the Archimedean case as well, see \S\ref{ssec:arc_pf} below.}

\medskip

$ $
%
\DrinI{In Theorem \ref{thm:futer} we are given $\RamiAS{\phi :X \to Y\times W}$ and $p:X\to K$. Let
  $\phi'=\phi\times p :X \to Y\times W \times K$. The idea is to apply Theorem \ref{thm:futerO} to $W\times K$ instead of $W$ and $\phi':X\to Y\times W \times K$ instead of $\phi :=X \to Y\times W$. Let
 $L' \subset T^*(Y\times W^{*} \times K)$ be the isotropic subvariety provided by Theorem \ref{thm:futerO} in this situation (in particular, $L'$ is stable under the homotheties of $W^*$). We can also assume that $L'$ is conic (otherwise replace $L'$ by its biggest conic subvariety).
Consider the embedding $ j:Y\times W^*=Y\times W^*\times\{ 1\}\hookrightarrow Y\times W^*\times F$. 
\RamiAE{Define $L$ to be the Zarizki closure\footnote{\RamiAE{In fact, $j^*(L')$ is closed, but this is not essential to us.}} of $j^*(L')$,}
 where $j^*$ has the same meaning as in Definition \ref{def:dir_and_inv_image}\eqref{def:dir_and_inv_image:2}.}
\RamiAC{By \RamiAF{Lemma \ref{lem:iso_dir_invs_im} {and Remark~\ref{r:easy_facts}},}
 $L$ is isotropic.}


 \DrinI{Let us show that $L$ has the property required in Theorem~\ref{thm:futer}.
  Let $F$ be a local field equipped with an embedding $K\hookrightarrow F$. Set $W_F:=W\otimes_KF$. The problem is to show that the wave front of the distribution
\begin{equation}  \label{e:ditsribution_in_question}
\RamiAB{\mu:=}\cF^*_{W_F}((\phi_F)_*((\psi \circ p_F) \cdot |\omega_F|))
\RamiAD{=\cF^*_{W_F}((\phi_F)_*(p_{F}^{*}(\psi ) \cdot |\omega_F|))}
\end{equation}
is contained in $L(F)$. By the definition of $L'$, the wave front of the distribution
\begin{equation}    \label{e:ditsribution_we_control}
\RamiAB{\mu':=}\cF^*_{W _F\times F}((\phi'_F)_*(|\omega_F|))
\end{equation}
is contained in $L'(F)$.

 }


\DrinI{First, let us show that
\begin{equation} \label{e:mu_and_eta}
\mu =\RamiAD{\mu'|_{Y(F)\times \RamiAE{W_F^{*}} \times \{1\}}}
\end{equation}
where 
the equality \eqref{e:mu_and_eta} is understood in the sense of
Definition\RamiAD{ \ref{def:cont_dep}. }
\DrinI{To this end, for each $t\in F$ consider the distribution
\[
\mu_t:=\cF^*_{W_F}((\phi_F)_*(p_{F}^{*}(\psi_{t})  \cdot |\omega_F|)),
\] where $\psi_{t}$ is the additive character of $F$ defined by $\psi_t(x)=\psi(tx)$. Note that $\mu_1=\mu$, so
\eqref{e:mu_and_eta} follows from the next lemma. }
\begin{lem}\label{lem:mu_t}
$ $
\begin{enumerate}
\item \label{lem:mu_t:1} $\{\mu_t\}_{t\in F}$ is a continuous family of distributions\footnote{{As usual, the space of distributions on $Y(F)\times W_F^*\,$ is equipped with the weak topology.}} on $Y(F)\times W_F^*\,$.

\item \label{lem:mu_t:2}
The distribution on $Y(F)\times W_F^* \times F$ corresponding to the family $\{\mu_t\}$ equals $\mu'$; \RamiAE{that is, for any $f\in \Sc(Y(F)\times W^{*}_F ,\C_{Y(F)} \boxtimes D_{W^{*}_F})$ and $g\in \Sc(F)$, we have
\begin{equation}
\langle \mu',f \boxtimes g\rangle=\int_{t\in F}\mu_t (f)g(t) dt.
\end{equation}
 }
\end{enumerate}
\end{lem}

\begin{proof}
Statement \eqref{lem:mu_t:1} is clear.
\RamiAE{Let us } prove \eqref{lem:mu_t:2}. \RamiAE{We have:
\RamiAD{
%
%
\begin{multline*}
 \left \langle \mu',f \boxtimes g\right\rangle  =\left \langle (\phi'_F)_*(|\omega_F|),\cF_{W_F} (f) \boxtimes \cF_{F} (g)\right\rangle=\left \langle |\omega_F|,(\phi'_F)^*(\cF_{W_F} (f) \boxtimes \cF_{F} (g))\right\rangle=\\
=\left \langle |\omega_F|,(\phi_F)^*(\cF_{W_F} (f)) \cdot p_{F}^*(\cF_{F} (g))\right\rangle=\left \langle |\omega_F|,(\phi_F)^*(\cF_{W_F} (f)) \cdot  p_{F}^*\left(\int_{t\in F}\psi_t \cdot g(t)dt\right)\right\rangle
\end{multline*}
On the other hand
\begin{multline*}
 \int_{t\in F} \left \langle\mu_t ,f \right\rangle g(t) dt
 =\int_{t\in F}\left \langle(\phi_F)_*(|\omega_F| \cdot p_{F}^*(\psi_t)),\cF_{W_F} (f)  \right\rangle g(t) dt=\\= \int_{t\in F}\left \langle(|\omega_F| \cdot p_{F}^*(\psi_t),(\phi_F)^*(\cF_{W_F}
(f))\right\rangle g(t)  dt =  \int_{t\in F} \left\langle |\omega_F|,(\phi_F)^*(\cF_{W_F} (f)) \cdot p_{F}^*(\psi_t)\right\rangle g(t)  dt
\end{multline*}
So it remains to prove that
$$\left \langle |\omega_F|,(\phi_F)^*(\cF_{W_F} (f)) \cdot  p_{F}^*\left(\int_{t\in F}\psi_t \cdot g(t)dt\right)\right\rangle=  \int_{t\in F} \left\langle |\omega_F|,(\phi_F)^*(\cF_{W_F} (f)) \cdot p_{F}^*(\psi_t)\right\rangle g(t)  dt$$
This follows from the fact that, for each particular $f$ and $g$, the integral can be replaced by a finite sum.
}
}
\end{proof}

 }

\DrinI{Thus we have proved
\RamiAE{\eqref{e:mu_and_eta}.}
By assumption, the wave front of $\mu'$ is contained in $L'(F)$.
So by {Corollary~\ref{c:Treves}}, to prove that the wave front of $\mu$ is contained in $L(F):=(\overline{j^*L'})(F)$, it suffices to check that $L'(F)$ satisfies the condition of
{Proposition~\ref{prop:Treves}}.
 In other words, we have to check that if $z\in \RamiAE{j}(Y\times \RamiAE{W^{*}})$ and
$\xi\in T^*_z(Y\times \RamiAE{W^{*}}\times K)$ are such that $(z,\xi )\in L'$ and $\xi$ is conormal to $k(Y\times W)$ then $\xi =0$.
Recall that $L'$  is assumed to be conic and stable under the action of the \RamiAH{multiplicative} group ${\mathbb G}_m$ on $Y\times \RamiAE{W^{*}}\times K$ that comes from  homotheties of $\RamiAE{W^{*}}\times K$; in other words, $L$ is stable under
${\mathbb G}_m\times {\mathbb G}_m\,$. So the tangent space to the $({\mathbb G}_m\times {\mathbb G}_m)$-orbit of $(z,\xi )$ has to be isotropic. This means that $\xi$ vanishes on the tangent space to the ${\mathbb G}_m$-orbit
of $z$. On the other hand, $\xi$ is assumed to be conormal to $\RamiAE{j}(Y\times \RamiAE{W^{*}})$. So $\xi=0$.
}

\section{\RamiT{The Archimedean case}} \label{sec:arc}
\RamiT{
In \S\S \ref{ssec:arc_dist} we recall the terminology relevant for the Archimedean case (in particular, the notion of partially Schwartz distribution).
In \S\S \ref{ssec:arc_pf} we
explain what should be added to the proof from \S\ref{sec:pf_main_bb}
 of non-Archimedean case
to make it valid in the Archimedean case \DrinI{(essentially, the only new ingredient is the elementary
Lemma~\ref{lem:dim_1_arch}).}

Throughout the section $F$ is an Archimedean field (i.e. $F$ is $\R$ or $\C$). Recall that we equip
$F$ with the  \emph{normalized} absolute value, which in case of $F=\C$ is the \emph{square} of the classical one.

\subsection{Distributions in the Archimedean case} \label{ssec:arc_dist}
Let $M$ be a smooth (real) manifold. Recall that the space $C^\infty_c(M)$ of test functions on $M$ is the space of smooth compactly supported functions endowed with the \RamiU{standard}
topology (\RamiU{recall that in this topology, a sequence} converges if and only if it has a compact joint support and converges uniformly with all its derivatives). Recall also that the space of distributions $C^{-\infty}(M,D_M)$ on $M$ is defined to be the dual of $C^\infty_c(M)$. Similarly, for any  \DrinI{smooth vector} bundle we can consider its smooth compactly supported sections and generalized sections.

\RamiAG{We will use the same notations as in \S\S\ref{ssec:dist} but we will replace $\Sc$ with $C^\infty_c$ and $\G$ with $C^{-\infty}$. The reason is that $\Sc$ and $\G$ stands for Schwartz, and in the non-Archimedean case Schwartz functions are just smooth compactly supported functions and Schwartz distributions are just distributions, unlike the Archimedean case.

\RamiAM{The content of \S\S\S\ref{sssec:fun_sp} and \S\S\S\ref{sssec:op_dist}}
hold\RamiAM{s} for the Archimedean case, with the obvious modifications (e.g. $l$-spaces are replaced with smooth manifolds and locally constant sheaves are replaced with  \DrinI{smooth vector} bundles). The statements of  \S\S\S\ref{sssec:wfs} hold with minor modifications.
\RamiAH{In particular,
{the role of Definition \ref{def:cont_dep} is played by the following one.}

\begin{definition}\label{def:cont_dep_arch}
Let $\xi\in C^{-\infty}(X \times Y)$ be a generalized function on a product
of analytic manifolds. We will say that $\xi$ depends continuously on $Y$ if

(i) for any $f \in C^{\infty}_c(X,D_X)$, the generalized function $\xi_f \in
C^{-\infty}(Y) $ given by
$\xi_f(g)=\xi(f \boxtimes g)$ is continuous;

(ii) for any $y \in Y$, the
functional $f \mapsto \xi_f(y)$ is continuous.

In this case we \DrinI{define} $\xi|_{X \times \{y\}} \in  C^{-\infty}(X
\times \{y\})$   by  $\xi|_{X \times \{y\}}(f):=\xi_f(y)$.
\end{definition}

{
\begin{rem}   \label{r:closed graph}
Using the closed graph theorem one can show that (i) implies (ii) (and moreover, (i) implies continuity of the map $C^{\infty}_c(X,D_X)\to C(Y)$ given by $f\mapsto\xi_f$). We will not need this fact.
\end{rem}
}

}
\RamiAH{
 We present the rest of the content of \S\S\S\ref{sssec:wfs},} with more details, for both  the Archimedean and the non-Archimedean case, in Appendix \ref{app:WF}.}

In order to discuss partial Fourier transform \RamiAG{as in  \S\S\S\ref{sssec:ft}} we will need to discuss test functions which are partially Schwartz.
\begin{definition}
Let $M$ be a smooth manifold and $V$ be a real vector space.
\begin{enumerate}
\item We define the space \RamiU{$C_{c}^{\infty,V}(M \times V)$} of partially Schwartz (along $V$) test functions on $M \times V$ to be the space of all smooth functions $f$ on $M \times V$ such that $\supp(f) \subset \RamiU{K} \times V$ for some compact $\RamiU{K \subset M}$, and for any polynomial differential operator $D$ on $V$ \RamiU{and any smooth differential operator $D'$ on $M$}, the function $\RamiU{D'}Df$ is bounded.
\item
\RamiU{We define a topology on this space in the following way. For any compact $K \subset M$, we let $C^{\infty,V}_K(M \times V)$ be the subspace of $C_{c}^{\infty,V}(M \times V)$ that consists of all functions 
supported on $K \times V$. We define the topology on $C^{\infty,V}_K(M \times V)$ by the semi-norms $f \mapsto |D'Df|$, where $D'$ and $D$ are as above. We define the topology on  $C_{c}^{\infty,V}(M \times V)$ to be the direct limit topology.}



\item \DrinI{Let $\xi$ be a distribution  on $M \times V$, i.e., a continuous linear functional on
$C^\infty_c(M\times V)$.
We say that  $\xi$ is \emph{partially Schwartz} along $V$ if this functional} can be continuously extended to the space of partially Schwartz test functions.
\item We can clearly extend the above definition to generalized sections of bundles of the type $E \boxtimes D_V$, where $E$ is a bundle on $M$.
\item We say that a generalized section of $E \boxtimes \C_V$, (where $\C_V$ is the constant bundle on $V$) is partially Schwartz if it becomes so after multiplication by a Haar measure on $V$.
\item The space of Schwartz generalized sections will be denoted by $C^{-\infty,V}(\dots)$.
\end{enumerate}
\end{definition}

\DrinI{Now the results of  \S\S\S\ref{sssec:ft} (with natural modifications) are valid (with the same standard proofs) for distributions which are partially Schwartz along the relevant vector space. Here is the precise formulation, whose only new ingredient is the fact that the operations performed on distributions preserve the partially Schwartz property.}

\begin{proposition} \label{prop:F_prop_arch}
Let $W,L$ be a real vector spaces and $X$ be a smooth manifold.
Let $\xi \in C^{-\infty}(X \times W,D_{X \times W})$, which is Schwartz along $W$.

\begin{enumerate}
\item
Let $U \subset X$ be an open set. Then 
{$\xi_{U\times W}$ is partially Schwartz along $W$ and} $\cF_{W}^{{*}}(\xi)|_{U\times W^{\RamiAQ{*}}}=\cF_{W}^{{*}}(\xi|_{U\times W})$.
\item 
{Let $\xi' \in C^{-\infty}(X \times W,D_{X \times W})$ and \RamiAS{l}et $X=\bigcup U_i$ be an open cover of $X$. Assume that $\xi'|_{U\times W}$ is partially Schwartz along $W$. Then $\xi'$ is  partially Schwartz along $W$.}
\item Let $f\in C^\infty(X)$ be a smooth function. Then 
{$f\xi$ is partially  Schwartz along $W$ and} $\cF_{W}^{{*}}(f\xi)=f\cF_{W}^{{*}}(\xi).$
\item Let $p:X \to Y$ be a proper map of smooth manifolds. Then 
{$p_*\xi$ is partially  Schwartz along $W$ and} $\cF_{W}^{{*}}(p_*\xi)=p_*\cF_{W}^{{*}}(\xi).$
\item Let $\nu$ be as  in Lemma \ref{prop:F_prop2} and  $\rho_\nu$  be as in Notation \ref{not:F_prop}
Then 
{the vertical arrows in the following diagram preserve the space of partially Schwartz functions and} it is commutative.

\begin{equation*}
\xymatrix
{        C^{-\infty,W^*}(X \times W^*)  \ar@{<-}^{\cF^*_{W}\quad}[r] \ar@{<-}^{(\rho_{{\nu^t}})^*}[d]& C^{-\infty,W}(X \times W,\C _X\boxtimes D_{W}) \ar@{<-}^{(\rho_{\nu})_* }[d]  \\
         C^{-\infty,L^*}(X \times L^*)  \ar@{<-}^{\cF^*_L\quad}[r]                                 & C^{-\infty,L}(X \times L,\C _X\boxtimes D_{L})
}
\end{equation*}
\end{enumerate}
\end{proposition}
%
%
\subsection{\DrinI{On the proofs of} the main results in the Archimedean case} \label{ssec:arc_pf}
$ $
\RamiY{The proof of Theorem \ref{thm:main_C_Arch}  follows the same lines as the proof of Theorem \ref{thm:main_bb}, but in each step we have to check that the distributions we consider are partially Schwartz along the relevant vector space. In other words we should prove part\RamiAE{s} \RamiAG{\eqref{thm:main_C_Arch:1} and \eqref{thm:main_C_Arch:2}} of Theorem \ref{thm:main_C_Arch}
together.
The reduction to Lemma \ref{lem:dim_1}
      is the same as  in  Theorem \ref{thm:main_bb}, but in Lemma  \ref{lem:dim_1} itself we need to be more careful. Namely, we have to precede it with the following lemma:}

 \begin{lemma}  \label{lem:dim_1_arch}
Let $Y$ be the affine space with coordinates $y_1, \dots, y_n.$ Let $p:Y \to F$ be defined by
$p= \prod_{i=1}^n y_i^{l_i}$,
where
$l_i \in \Z_{\geq 0}$. Let $\omega$ be the top differential form on $Y$
given by $\omega= (\prod_{i=1}^n y_i^{r_i}) d{y_1} \wedge \dots \wedge d{y_n}$, where $r_i \in \Z$.
 Suppose $r_i \geq 0$ whenever $l_i=0$,
so $\omega$ is regular on the set $Y_0:=\{ y\in Y | p(y)\neq 0\}$.
Define $i:Y_0\hookrightarrow Y \times F$ by $i(y):=(y,p(y)^{-1})$.
 Then the distribution $i_*(|\omega|)$ is Schwartz, \RamiU{and in particular it is partially Schwartz} along~$F$.
\end{lemma}



\begin{proof}
Consider the scalar product of $i_*(|\omega|)$ against $f \in C_c^\infty(Y \times F)$. It suffices to get for it an estimate of the form
\begin{equation}   \label{e:form_of_estiamate}
|i_*(|\omega|)(f)|\le \sup_{(y,x)\in Y\times F} |u(y,x)f(y,x)|,
\end{equation}
where $u$ is some polynomial on $Y\times F$.

For brevity, write $y$ instead of $(y_1,\dots,y_n)$ and $y^r$ instead of $\prod\limits_{i=1}^n y_i^{r_i}$.
Set
$$s(y):= \prod\limits_{i=1}^n (1+|y_i^{2}|).$$
We have
\begin{multline*}
|i_*(|\omega|)(f)|=|\int_{Y_0} y^{r} f(y,p(y)^{-1})dy |\leq\\
\leq \int_{Y_0} |y^{r}| \cdot|f(y,p(y)^{-1})|dy\leq C\cdot \sup_{y\in Y_0} s(y)\cdot |y^r|\cdot |f(y,p(y)^{-1})|,
\end{multline*}
where $C:= \int\limits_{Y}  s(y)^{-1}dy$.

The conditions on $l_i$ and $r_i$ imply that for $N$ big enough the function $q(y):=p(y)^N\cdot y^r$ is a polynomial. We have
\begin{multline*}
 \sup_{y\in Y_0} s(y)\cdot |y^r|\cdot |f(y,p(y)^{-1})|=\sup_{y\in Y_0} s(y)\cdot |q(y)|\cdot |p(y)|^{-N}\cdot |f(y,p(y)^{-1})|\le\\
 \le
\sup_{(y,x)\in Y\times F} s(y)\cdot |q(y)|\cdot |x^N|\cdot |f(y,x)|.
\end{multline*}
Thus we get an estimate of the form \eqref{e:form_of_estiamate}.
\end{proof}
}
\RamiY{Theorems \ref{thm:futerO} and \ref{exp} and Corollary \ref{cor:exp} are also proven in the same way as in the non-Archimedean case. So we are left with Lemma \ref{lem:Sch2} and Theorem  \ref{thm:futer}.}

\RamiAG{In fact, we will need a slightly stronger version of Lemma \ref{lem:Sch2}.
For its formulation we will need the following notion.\footnote{{In connection with
Definition~\ref{d:strictly continuous}, see Definition~\ref{def:cont_dep_arch} and Remark~\ref{r:closed graph}.}}
\begin{defn}   \label{d:strictly continuous}
Let $V$ be a real vector space and $X,Y$ be  smooth manifolds.
We call a family of generalized functions $\xi_t \in C^{-\infty,V}(Y \times V)$ parameterized by $t \in X$ strictly continuous if it gives rise to a continuous map $C_c^{\infty,V}(Y \times V,D_{Y \times V}) \to C(X)$\RamiAH{, where the topology on $C(X)$ is the open compact one}.
\end{defn}
The following lemma is a  stronger version of Lemma \ref{lem:Sch2}:
\begin{lem} \label{lem:Sch3}
{
In the situation of  Lemma \ref{lem:Sch2}, set
$$\xi_t:= (\phi_F)_*((\psi_t \circ p_F) \cdot |\omega_F|)\in C^{-\infty}(Y(F) \times W_F,D_{Y(F) \times W_F}),$$
where $t\in F$ and $\psi_{t}$ is the additive character of $F$ defined by $\psi_t(x)=\psi(tx)$. Then each $\xi_t$
is partially Schwartz along $W_F$ and the family of distributions $\xi_t$, $t\in F$, is  strictly continuous.
}
\end{lem}
}

\RamiAG{In order to prove this lemma} we will need the following \RamiAG{one}:
\begin{lem}\label{lem:dir_is_Sch}
Let $Y$ be \RamiAG{an algebraic} manifold and let $V_1,V_2$ be finite dimensional $F$-vector spaces. Choose a  \RamiAG{Haar measure} on $V_2$.  Let $Z\subset Y \times V_1\times V_2$ be an algebraic subvariety such that the projection of $Z$ \RamiAG{to} $Y \times V_1$ is proper (and hence finite). Let $\xi$ be a distribution on $Y(F) \times V_1\times V_2$ which is Schwartz along $V_1\times V_2$ and supported on $Z(F)$. Let $p: Y \times V_1\times V_2 \to  Y \times V_1$ be the projection. Then $p_*(\xi)$ is Schwartz along $V_1$.

\RamiAG{
Moreover, if $\xi_t \in C^{-\infty,V_1\times V_2}(Y(F) \times V_1\times V_2)$ is a strictly continuous family of distributions which are supported on $Z(F)$, then $p_*(\xi_{t}) \in C^{-\infty,V_1}(Y(F) \times V_1)$ is a strictly continuous family of distributions.
}
\end{lem}
\RamiAG{
For the proof we will need the following lemma:
\begin{lem}\label{lem:ex_poly}
Let $Y$ be an affine algebraic manifold and $V$ be a finite dimensional $F$-vector space. Let $Z\subset Y \times V$ be an algebraic subvariety such that the projection of $Z$ to $Y$ is proper (and hence finite). Then there exists a real polynomial $p$ on $Y$ and a norm $\| \cdot \|$ on $V$ such that for any $(y,v) \in Z(F)$, we have {$\max (\| v \|,1) \le p(y)$.}

%
\end{lem}

\begin{proof}
Let $z_i$ be the coordinates on $V$. Since the projection of $Z$ to $Y$ is  finite, we can find polynomials $\{a_{ij}\}_{j=1\dots N_i}$ on $Y$ such that $(z_i)^{N_i+1}+\sum_{j=1\dots N_i} a_{ij}(y)(z_i)^j=0$ for all $(y,z)\in Z(F)$. This easily implies the assertion.
\end{proof}

\begin{proof}[Proof of Lemma \ref{lem:dir_is_Sch}]
{We can assume that $Y$ is affine.}
By Lemma \ref{lem:ex_poly}, we can find a real polynomial $p$ on $Y \times V_1$ such that for any $(y,v_1,v_2)\in Z(F)$, we have 
{$\max (\| v \|,1) \le p(y)$.} Let $\phi$ be a smooth  function on $\R$ such that $\phi([-1,1])=1$ and $\phi(\R-[-2,2])=0$. Let $f\in C^\infty(Y \times V_1\times V_2)$ be defined by
$f(y,v_1,v_2)=\phi(\| v_2 \|/p(y,v_1)).$ Let $$pr:X \times V_1 \times V_2 \to X \times V_1$$ be the projection. Define $$pr^*_p:C_c^{\infty}(X \times V_1,D_{X \times V_1})\to C_c^{\infty}(X \times V_1 \times V_2,D_{X \times V_1 \times V_2})$$ by $pr^*_p(g)=pr^*(g)\cdot f$. It is easy to see that $pr^*_p$  can be continuously extended to a map $$C_c^{\infty,V_1}(X \times V_1)\to C_c^{\infty,V_1\times V_2}(X \times V_1 \times V_2)$$   and that for any $g\in C_c^{\infty}(X \times V_1)$ and $\xi \in C_c^{-\infty}(X \times V_1 \times V_2)$, we have:
$$\langle \xi, pr^*_p(f)\rangle=\langle pr_*(\xi) ,f \rangle.$$ This proves the assertion.
\end{proof}

Now we can deduce Lemma \ref{lem:Sch3} from Lemma \ref{lem:dir_is_Sch} and Theorem  \ref{thm:main_C_Arch}\eqref{thm:main_C_Arch:1}.
\begin{proof}[Proof of Lemma \ref{lem:Sch3}]
Let   $\phi'=\phi\times p :X \to Y\times W \times K$ and $\xi':=(\phi'_F)_*(|\omega_F|)$. By Theorem \ref{thm:main_C_Arch}\eqref{thm:main_C_Arch:1} the distribution $\xi'$ is partially Schwartz with respect to $W \times F$. For any $t \in F$, let $\xi'_t :=\xi' \cdot 1_{Y(F) \times W_F} \boxtimes \psi_t$. It is easy to see that $\xi'_t$ is a strictly continuous family of partially Schwartz distributions and $\pr_*(\xi'_t)=\xi_t$. Lemma \ref{lem:dir_is_Sch} now implies the assertion.
\end{proof}

Now let us prove Lemma \ref{lem:mu_t}  in the Archimedean case. The distributions $\mu_t$ and $\mu'$ from Lemma~\ref{lem:mu_t} can be written as
\[
\mu_t=\cF_{W_F}^*(\xi_t) ,\quad \mu'=\cF_{W_F}^*(\eta ) ,
\]
where
\begin{equation}  
\xi_t:=(\phi_F)_*((\psi_t \circ p_F) \cdot |\omega_F|), \quad t\in F,
\end{equation}
\begin{equation}    
\eta:=\cF_F^*((\phi'_F)_*(|\omega_F|)),
\end{equation}
and $\phi':X \to Y\times W \times K$ is defined by $\phi'=\phi\times p$.
By Lemma~\ref{lem:Sch3}, each $\xi_t$ is partially Schwartz along $W_F$ and the family of distributions
$\{\xi_t\}$ is  strictly continuous. So each $\mu_t$ is a well-defined distribution and the family $\{\mu_t\}$ is  continuous. This proves Lemma~\ref{lem:mu_t}(1). It is easy to check that the distribution on
$Y(F)\times W_F \times F$ corresponding to the family $\{\xi_t\}$ equals $\eta$. By strict continuity of
$\{\xi_t\}$, this implies Lemma~\ref{lem:mu_t}(2), which says that the distribution on
$Y(F)\times W^*_F \times F$ corresponding to the family $\{\mu_t\}$ equals $\mu'$.

Theorem \ref{thm:futer} is deduced from Lemma~\ref{lem:mu_t} just as in the non-Archimedean case.
}

\appendix
\section{The wave front set} \label{app:WF}
\setcounter{lemma}{0}
In this section we give an overview of the theory of the wave front set \RamiAH{as developed in \cite{Hor} for the Archimedean case and in \cite{Hef} for the non-Archimedean case.

We will discuss these two cases simultaneously. We will discuss the wave front set of general distributions which are functionals on smooth compactly supported functions. We will use the notations $C^{-\infty}$ and $C_c^{\infty}$ for the spaces of generalized functions and test functions as in \S\S\ref{ssec:arc_dist}. Note that in the non-Archimedean case, there is no difference between Schwartz functions and smooth compactly supported functions, and between general distributions and Schwartz distributions.
}

We explain here the results that  we quote in \S\ref{sssec:wfs}. We give an explicit  reference for some of them and provide proofs for the others.

\begin{defn}\label{def:wf}$ $
\begin{enumerate}
\item
Let $V$ be an $F$-vector space, \RamiAH{with} $\dim V<\infty$.
Let $f \in C^{\infty}(V^*)$ and $w_0 \in V^*$. We  say that $f$ \emph{vanishes asymptotically in the direction of} $w_0$
if there exists $\rho \in \RamiAH{C^\infty_c}(V^*)$ with $\rho(w_0) \neq 0$ such that the function $\phi \in C^\infty(V^* \times F)$ defined by $\phi(w,\lambda):=f(\lambda w) \cdot \rho(w)$ is a Schwartz function.

\item
Let $U \subset V$ be an open set and $\nu \in \RamiAH{C^{-\infty}(U,D_U)}$. Let $x_0 \in U$ and $w_0 \in V^*$.
We say that $\nu$ is \emph{smooth at} $(x_0,w_0)$ if there exists a compactly supported \RamiAH{non-negative} function $\rho \in \RamiAH{C^\infty_c} (V)$ with $\rho(x_0)\neq 0$ such that $\cF^*(\rho \cdot \nu)$ vanishes asymptotically in the direction of
$w_0$.
\item
The complement in $T^*U$
of the set of smooth pairs $(x_0, w_0)$ of $\nu$ is called the
\emph{wave front set of} $\nu$ and denoted by $WF(\nu)$.
\end{enumerate}
\end{defn}

\DrinF{\begin{remark}
Let $WF_H (\nu )$ denote the wave front set defined by L.~H\"ormander \cite[Definition 8.1.2]{Hor} for $F=\R$ and by D.~Heifetz \cite{Hef} for non-\RamiP{Archimedean} fields $F$. Let us explain the relation between
$WF_H (\nu )$ and $WF (\nu )$. First of all, $WF_H (\nu )$ is a subset of $T^*U -(U \times \{ 0\})$
stable under multiplication by $\lambda\in\Lambda$, where $\Lambda \subset F^{\times}$ is
some open subgroup (the definition of $WF_H$ from \cite{Hef} explicitly depends on a choice of
$\Lambda$, H\"ormander always takes $\RamiAS{\Lambda}=\R_{> 0}$)\RamiO{, However it is not necessarily stable under multiplication by
 $F^{\times}$}. Second,
\begin{equation}\label{eq:WF_relation}
WF(\nu)-(U \times \{0\})= F^{\times}\cdot WF_H(\nu).
\end{equation}
To prove \eqref{eq:WF_relation} for $F=\R$, one needs the following observation.
In Definition \ref{def:wf}(2) we require not only the function $\cF^*( \rho \cdot \nu )$ to rapidly decay at $\infty$ but also the same property for $D \cF^*( \rho \cdot \nu )$, where $D$ is any differential operator with constant coefficients. However, it suffices to require the rapid decay of
$\cF^*( \rho \cdot \nu )$ (as in \cite[Definition 8.1.2]{Hor}): the rest follows from
\cite[Lemma 8.1.1]{Hor} combined with the formula $D \cF^*( \rho \cdot \nu )=\cF^*(p\cdot \rho \cdot \nu )$, where $p$  is the polynomial corresponding to $D$.
\end{remark}
}




The following lemma is trivial.
{
\begin{lemma} \label{lem:WF_prop_aff}

Proposition \ref{prop:WF_prop} \RamiAH{\eqref{prop:WF_prop:1}-\eqref{prop:WF_prop:4}}
holds for the case \RamiAH{when} $X\RamiAH{\subset}F^n$ \RamiAH{is an open set} and $\xi \in \RamiAH{C^{-\infty}(X,D_X)}$. Namely:
\begin{enumerate}
\item $P_{T^*(X)}(WF(\xi))=WF(\xi) \cap (X) =\Supp(\xi).$
\item \RamiAS{$WF(\xi) \subset X$ {if and only if }}$\xi$ \RamiE{is smooth}.
\item Let $U \subset X$ be an open set. Then $WF(\xi|_U) = WF(\xi) \cap T^*(U).$
\item Let $\xi' \in \RamiAH{C^{-\infty}(X,D_X)}$ and $f,f' \subset C^\infty(X)$. 
Then
$$WF(f \xi+f' \xi') \subset WF(\xi) \cup
WF(\xi').$$
\end{enumerate}

\end{lemma}
}

%

\begin{corollary}
For any locally constant sheaf \RamiAH{(or, in the Archimedean case, a vector bundle)}  $E$ on \RamiAH{$X$,} we can define the wave
front set of any element in $\RamiAH{C^{-\infty}}(\RamiAH{X},E)$.
{Moreover, the last lemma (Lemma \ref{lem:WF_prop_aff}) will hold in this case, too.}
\end{corollary}

\begin{proposition}
[\RamiAH{see \cite[Theorem 8.2.4]{Hor} and  \cite[Theorem 2.8.]{Hef}}]
\label{submrtion}
Let $U \subset F^m$ and $V \subset F^n$ be open subsets, and
suppose that  $f: U \to V$ is an analytic submersion. Then for any
$\xi \in \G(V)$, we have $WF(f^*(\xi)) \subset f^*(WF(\xi))$.
\end{proposition}

\begin{corollary} \label{iso}
Let $V, U \subset F^n$ be open subsets and $f: V \to U$ be an
analytic isomorphism. Then for any $\xi \in \G(V)$, we have
$WF(f^*(\xi)) = f^*(WF(\xi))$.
\end{corollary}

\begin{corollary}\label{cor:bundle}
Let $X$ be an analytic manifold, $E$ be a locally constant sheaf \RamiAH{(or, in the Archimedean case, a vector bundle)}
on $X$ . We can define the wave front set of any element in
$\Sc^*(X,E)$ and $\G(X,E)$. Moreover, {Lemma \ref{lem:WF_prop_aff} and Proposition \ref{submrtion} hold for this case.}
\end{corollary}

%

\begin{proposition}
Proposition \ref{prop:WF_prop} \RamiAH{\eqref{it:G_act}}  holds.

\RamiAH{Namely,
  let} $G$ be an analytic group acting on an analytic manifold $X$ and a locally constant sheaf \RamiAH{(or, in the Archimedean case, a vector bundle)}   $E$ over it. Suppose $\xi \in \RamiAH{C^{-\infty}}(X,E)$ is
$G$-{invariant}. Then $$ WF(\xi) \subset \{(x,v) \in T^*X(F)|v(\g x)=0\}$$
where $\g$ is the Lie algebra of $G$.
\end{proposition}
\RamiAH{
\begin{proof}
In the non-Archimedean case, this is Theorem 4.1.5  of \cite{Aiz}.
In the  Archimedean case, the same proof works.
\end{proof}
}
The following  proposition is {essentially} proved in \cite{Gab} for the Archimedean case.
 We include its proof
here for completeness.
\begin{proposition} \label{prop:wf_push}
Let $p:X \to Y$ be {an analytic map}
and \RamiL{$E$} be a locally constant sheaf \RamiAH{(or, in the Archimedean case, a vector bundle)}  over \RamiK{$Y$}.
 Let $\xi \in \RamiAH{C^{-\infty}(X,p^{\RamiAN{!}}(E))}$\RamiAN{, where $p^{!}$ is the pullback twisted by relative densities (see \S\S\S\ref{sssec:op_dist}\eqref{it:dist_notstd2})}. 
Assume $p|_{\supp (\xi)}$ is proper.
Then
 $$\WF(p_*(\xi)) \subset p_*(\WF(\xi)) $$
\end{proposition}
\RamiAH{
For the proof we will need the following lemma:
\begin{lemma}\label{lem:WF_mon}
Let $V$ be an $F$-vector space, with $\dim V<\infty$. Let $\xi \in C^{-\infty}(V,D_V)$ be a compactly supported distribution. Assume $\cF(\xi)$ vanishes asymptotically along $v \in V^*$. Then for any $\rho \in C_c^{\infty}(V)$, the function $\cF(\rho \cdot \xi)$ vanishes asymptotically along $v$.
\end{lemma}
\begin{proof}
In the \RamiAS{Archimedean} case, this is  Lemma 8.1.1 from \cite{Hor}. In the non-\RamiAS{Archimedean} case, it is obvious.
\end{proof}
}

\begin{proof}[\RamiAH{Proof of Proposition \ref{prop:wf_push}}]
$ $
\begin{enumerate}[{Case} 1.]
 \item  $p$ is a submersion.\\
Without loss of generality, we may assume \RamiAH{that $E=D_Y$,} 
$Y=F^k$ and $X=Y \times D$, where $D \subset F^n$ is a standard
\RamiAH{open}
 poly-disk, \RamiAH{$\supp(\xi) \subset  Y \times D'$ where  $D' \subset D$ is a closed poly-disk} and $p$ is the projection. Let $y \in Y$. Let $v \in T^*_y{Y}$ be such that for any $x \in D$, we have $((y,x),(d_{(y,x)}(p))^*(v)) \notin \WF(\xi)$. We have to show that $p_*(\xi)$ is smooth at $(x,v)$.

By the definition of $WF(\xi)$, we can find for any $x \in \RamiAH{D}$  \RamiAH{a non-negative function $f_x\in C_c^\infty(X)$ such that $f_x((y,x))\neq 0$ and $\cF(f_x \cdot \xi)$ vanishing asymptotically in the direction of $(0,v)$.
{So} we can construct a non-negative function $f\in C_c^\infty(X)$ such that $f$ does not vanish on
{$p^{-1}(y)\cap (Y\times D')$} and $\cF(f \cdot \xi)$ vanishes asymptotically in the direction of $(0,v)$. By Lemma \ref{lem:WF_mon}, we may assume that {$f$ has the form $f(y',x')=g(y')$ for
some $g \in C_c^\infty(Y)$}. This implies that $\cF(g \cdot p_*(\xi))$ vanishes asymptotically in the direction of $v$.
}

%



\item  $p$ is a closed embedding.\\
Without loss of generality, we may \RamiM{assume  $E$} is trivial,
 $X=D^k$ \RamiM{and} $Y =X \times D^n$, where $D \subset F$ is a  disk and $p$ is the standard embedding.
In this case the assertion is obvious.

\item  The general case.\\
It follows from the previous cases by decomposing $p=pr \circ \gamma$, where $\gamma:X \to X \times Y$ is the graph embedding
and $pr:X \times Y \to Y$ is the projection.

\end{enumerate}
\end{proof}
\subsection{\RamiAM{Pullback of distributions}}
\RamiAM{In order to discuss pullback of
distributions under general maps,}%
\RamiAH{ we need to define the topology on the space $C_\Gamma^{-\infty}(X,E)$ of generalized sections whose wave front set is included in $\Gamma \subset T^*(X)$.
\begin{defn}
We will define the topology in terms of converging sequences rather than open sets, but one can easily modify this definition in order to get an actual definition of topology.
 Let us first define some auxiliary topologies on some related spaces.
\begin{enumerate}
 \item In the Archimedean case, the space of Schwartz functions on a vector space $V$ is equipped with a well known \Fre topology. In the non-Archimedean case, we say that a sequence of Schwartz functions  converges if all its elements are in the same finite dimensional vector space and it converges there.
\item We say that a sequence of functions $f_i$ in the space $C^{\infty}_{v_0}(V)$ of smooth functions on $V$ which vanishes asymptotically along ${v_0}$ converges if there exists a  $\rho \in C^\infty_c(V)$ with $\rho(v_0) \neq 0$ such that the sequence of functions $\phi_i \in C^\infty(V \times F)$ defined by $\phi(v,\lambda):=f_i(\lambda v) \cdot \rho(w)$ converges in the $\Sc(V)$.
\item We say that a sequence of distributions $\xi_i \in C^{-\infty}_\Gamma(V,D_v)$  converges if it weakly converges and for any $(x,w)\notin \Gamma,$ there exists a function $\rho \in C^\infty_c(V)$ with $\rho(x) \neq 0$ such that the sequence $\cF(\rho \cdot \xi_i) \in C^{\infty}_{w}(V^*)$ converges.
\item This easily defines  a topology on $C^{-\infty}_\Gamma(X,E)$, for any analytic variety $X$ and a locally constant sheaf (or in the Archimedean case, a vector bundle) on $X$.
\end{enumerate}
\end{defn}

\begin{proposition}[{\cite[Theorem 8.2.4.]{Hor} and  \cite[Theorem 2.8.]{Hef}}]
\label{prop:2WF_prop3}
Let $p: Y \to X$ be an analytic map of analytic manifolds, and let
$$N_p= \{(x,v)\in T^{*}X| x = p(y)  \text{ and } d_{\RamiAB{y}}^* p(v)=0 \text{ for some } y\in Y \}.$$
Let $E$ be a locally
constant sheaf (or, in the Archimedean case, a vector bundle) on $X$.
Let $\Gamma \subset T^{*}X$ be a conic closed subset such that $\Gamma \cap N_p \subset X$.

Then the  map $p^{*}:C^\infty(X,E)\to C^\infty(Y,p^*(E))$
 has a unique continuous extension
to a
map
$p^{*}:C^{-\infty}_\Gamma(X,E)\to C^{-\infty}(Y,p^*(E))$.
 Moreover,
for any $\xi \in C^{-\infty}_\Gamma(X,E)$, we have:
$WF(p^*(\xi)) \subset p^*(WF(\xi))$.
\end{proposition}
}

\begin{rem}\label{rem:iver_im}
Here is an explicit procedure to compute $p^*(\xi)$: we may assume that $X$ is a vector space and $E$ is trivial. Let $f_1\in C^{\infty}_c(X)$ and $f_2\in C^{\infty}_c(X,D_X)$ such that $f_1(0)=1$ and $\int f_2=1$. Let $\lambda_{i} \in F$ be a sequence that converges to infinity. Let $\xi_i:= \rho_{\lambda_{i}}(f_{1}) \cdot (\rho_{\lambda_{i}^{-1}}(f_{2}) * \xi)$, where $\rho_{\lambda_{i}}$ is the homothety as defined in Notation \ref{not:F_prop}. Note that $\xi_i$ are smooth and compactly supported functions. Now, $p^*(\xi)$ is the weak limit of $p^*(\xi_i)$.
\end{rem}

{
Now we can prove Proposition~\ref{prop:Treves}. First, let us recall its formulation.
Let $$\xi\in \G(X \times Y)$$ be a generalized function on a product of analytic manifolds. Assume that
$\xi$ depends continuously on $Y$, so for each $y\in Y$ we have the generalized function $\xi|_{X \times \{y\}}$
from Definition~\ref{def:cont_dep}. Assume also that $WF(\xi) \cap CN_{X \times \{y\}}^{X \times Y}\subset X \times Y$, so for each $y\in Y$ we have the pullback $j_y^*(\xi )$ in the sense of
Proposition~\ref{prop:2WF_prop3}, where
$j_y:X \times \{y\}\hookrightarrow X \times Y$ is the embedding. Proposition~\ref{prop:Treves} says that in this situation
$$j_y^*(\xi)=\xi|_{X \times \{y\}}\, .$$
To prove this equality, it suffices to compute $j_y^*(\xi)$ using Remark~\ref{rem:iver_im} and choosing $f_1, f_2\in C^{\infty}_c(X\times Y)$ to be compatible with the product structure on $X \times Y$.
}


\section{Symplectic geometry of the co-tangent bundle} \label{app:symp_geom}
In this section we \RamiAH{provide a proof of the facts from the symplectic geometry of the co-tangent bundle that we used in \S\S\ref{sec:lag}}\RamiAS{.}

\subsection{\RamiAH{Images of isotropic {subsets}}}
\RamiAH{
We will prove Lemma \ref{lem:iso_dir_invs_im} using  Remark \ref{rem:simp_gem}. For this, we will need the following notion:

\begin{defn}
Let $M$, $N$ be symplectic manifolds. A \hrefHid{http://ncatlab.org/nlab/show/Lagrangian+correspondence}{Lagrangian correspondence}  (or, respectively, isotropic correspondence) between them is a subvariety $L \subset M \times N$ which is Lagrangian (or, respectively, isotropic) with respect to the symplectic form $\omega_M \oplus (-\omega_N)$.
\end{defn}

Lemma \ref{lem:iso_dir_invs_im} follows now from the next one:
\begin{lemma}\label{lem:lag_cor}
$ $
 \begin{enumerate}
  \item \label{lem:lag_cor:1} Let $\phi:X_1 \to X_2$ be a morphism of manifolds. Then the correspondence $\Lambda_\phi \subset T^*(X_1) \times T^*(X_2)$ described in Remark \ref{rem:simp_gem} is Lagrangian.
  \item \label{lem:lag_cor:2} Let $M$, $N$ be algebraic symplectic manifolds. Let $L \subset M \times N$ be an isotropic  correspondence  between them. Let $I \subset M$  be an isotropic {constructible subset}. Then the  {constructible subset} $L(I) \subset N$  is also isotropic.
 \end{enumerate}
\end{lemma}

{Both statements are well known. For the second one, see, e.g.,
\cite[Prop.~ 2.7.51]{CG} or \cite[Lemma~1]{G}. To prove the first one, note that the symplectic form on
$T^*(X_i)$ is the differential of the canonical 1-form $\eta_i$ on $T^*(X_i)$ and that the pullbacks of $\eta_1$  and $\eta_2$ to $\Lambda_\phi \subset T^*(X_1) \times T^*(X_2)$ are equal to each other.}

}

\subsection{\RamiAH{Equivalent definitions of isotropic  {subsets}}}
Let us now prove Lemma \ref{lem:isotr_discrip}.
%
%
%

Clearly (\ref{4}) $\Rightarrow$ (\ref{3}) $\Rightarrow$ (\ref{2}). By Proposition \ref{prop:sub_iso},
(\ref{2}) $\Rightarrow$ (\ref{1}).


Now \RamiAS{it remains} to show that \RamiK{(\ref{1})} $\Rightarrow$  (\ref{4}). We have to show that any isotropic $C \subset T^*X$
is contained in a union as in (\ref{4}). We will prove it by induction on $\dim (C)$ and $\dim P_{T^*X}(C)$\RamiAM{, where $P_{T^*(X)}:T^*(X)\to X$ is the projection}.
Let $C'$ be the set of smooth points in $C.$ Note that $\dim(C-C')< \dim (C)$ and by Proposition \ref{prop:sub_iso},
 $C-C'$ is isotropic. Thus, by induction, we may assume that  $C-C'$ satisfies (\ref{4}). Therefore it is enough
 to prove that $C'$ \RamiA{satisfies} (\ref{4}).

Consider the map $q:=P_{T^*X}|_{C'} $ as a map from $C'$ to
$\overline{P_{T^*X}(C')}$.   Let $U \subset \overline{P_{T^*X}(C')}$
be the  \RamiM{set of} those smooth points of $\overline{P_{T^*X}(C')}$ which are regular values of $q$,
 and let $C''=q^{-1}(U)$. Note that by the algebraic
 Sard lemma,
 $\dim(P_{T^*X}(\overline{C'-C''})){=\dim(\overline{P_{T^*X}({C'})}-U)}< \dim (P_{T^*X}(C')) =\dim (P_{T^*X}(C))$
 and by Proposition \ref{prop:sub_iso}, $\overline{C'-C''}$ is isotropic.
 Thus, by induction, we may assume that  $\overline{C'-C''}$ (and thus also ${C'-C''}$)
satisfies (\ref{4}). Therefore it is enough to prove that $C''$  satisfies (\ref{4}).
 We will prove that $C'' \subset {CN_{U}^X}$. For this let $x \in U $  and let $Y:=q^{-1}(x) \subset T^*_x X$.
 Fix any $y \in Y$. We know that $T_yC''$ is isotropic, i.e. $T_yC'' \bot (T_yC'')$.
Thus  we have $T_y Y = \ker d_y q \subset  (\Im d_y q)^\bot= CN_{U,x}^X.$ This implies that any connected
 component of $Y$ is a subset of a \RamiAH{shift} of $CN_{U,x}^X$. Since $\overline{Y}$  is conical, this shows that
 $Y \subset CN_{U,x}^X$.

\subsection{\RamiAH{Co-normal bundle to a subbundle.}}
\RamiAH{Finally, let us prove Lemma \ref{l:symplectic_lem}. First recall its formulation:
}
\RamiCPE{
\begin{lemma}    \label{l:symplectic_lem_pf}
Let $V$ be a finite dimensional vector space over $F$ and $X$ be a manifold.
Let $E$ be a vector bundle over $X$ which is a subbundle of the trivial vector bundle $X \times V$.
Then $CN_{E^\bot}^{X \times V^*}=CN_E^{X \times V}$.

Here
$$CN_E^{X \times V} \subset T^*(X \times V)=T^*(X) \times V \times V^*,$$
$$CN_{E^\bot}^{X \times V^*} \subset T^*(X \times V^*)=T^*(X) \times V^* \times V,$$
and the symplectic manifolds $T^*(X) \times V \times V^*$ and $T^*(X) \times V \times V^*$ are identified via the map
$$V \times V^*{\buildrel{\sim}\over{\longrightarrow}} V^* \times V, \quad (v,w)\mapsto (w,-v).$$
\end{lemma}

\begin{proof}
Set $L:=CN_E^{X \times V}$. Without loss of generality we may assume that $X$ is irreducible. Then so is $L$.
Clearly $L$ is a closed Lagrangian submanifold of
$T^*(X \times V)=T^*(X) \times V \times V^*$. 
{It is easy to check that the image of $L$ in $X \times V^*$ equals $E^\bot$. So it remains to show that
$L$ is conic as a submanifold of $T^*(X \times V^*)$. In terms of the action of
${\mathbb{G}}_m^3$ on $T^*(X) \times V \times V^*$, we have to show that $L$ is stable with respect to
the subgroup
\begin{equation}   \label{e:the_subgroup}
\{(\lambda, \lambda ,1)\,|\,\lambda\in {\mathbb{G}}_m\}\subset{\mathbb{G}}_m^3 \, .
\end{equation}

Since $E\subset X\times V$ is
${\mathbb{G}}_m$-stable the submanifold $L$ is stable with respect to the subgroup
\[
\{(1,\lambda,\lambda^{-1})\,|\,\lambda\in {\mathbb{G}}_m\}\subset{\mathbb{G}}_m^3\, .
\]
Clearly $L$ is conic as a submanifold of $T^*(X \times V)$, which means that $L$ is stable with respect to the subgroup
\[
\{(\lambda, 1, \lambda )\,|\,\lambda\in {\mathbb{G}}_m\}\subset{\mathbb{G}}_m^3\, .
\]
But $(\lambda, \lambda ,1)=(1,\lambda,\lambda^{-1})\cdot (\lambda, 1, \lambda )$,
so $L$ is stable with respect to the subgroup~\eqref{e:the_subgroup}.
}
\end{proof}

\begin{rem}
Here is a sketch of a slightly different proof of Lemma~\ref{l:symplectic_lem_pf}. The subbundle $E$ defines a map
\[
f:X\to \{\mbox{the Grassmannian of all subspaces of } V\}.
\]
Its differential at $x\in X$ is a linear map $T_xX\to\Hom (E_x,V/E_x)=\Hom (E_x,(E_x^\bot)^*)$; it defines a
bilinear map $B_x:E_x\times E_x^\bot\to T_x^*X$.
One checks that $CN_E^{X \times V}$ has the following description in terms of $B_x$:
let $(x,\xi )\in T^*X$, $v\in V$, $w\in V^*$, then
\begin{equation}
(x,\xi ,v,w)\in CN_E^{X \times V} \Leftrightarrow v\in E_x, \, w\in E_x^\bot , \, \xi =-B_x (v,w).
\end{equation}
Lemma~\ref{l:symplectic_lem_pf} follows from this description.
\end{rem}

\begin{rem}
Lemma \ref{l:symplectic_lem_pf} is closely related to the following fact: if $\xi\to X$ is any vector bundle and $\xi^*\to X$ is the dual bundle then
there is a canonical symplectomorphism between the cotangent bundles of $\xi$ and $\xi^*$ (see
\cite[Theorem 5.5]{MX} and also  \cite[Section 3.4]{Ro}).
\end{rem}
}

\end{document}